\documentclass[11pt]{amsart} 
\usepackage{xspace,amssymb,amsmath,amscd,amsthm,epsfig,etoolbox, hyperref, ytableau, float, mathrsfs, mdframed, enumerate, tikz, mathtools, bm, braket, tikz, multicol}
\usetikzlibrary{calc, tikzmark}
\hypersetup{colorlinks = true, citecolor = blue, linkcolor = blue}
\usepackage[capitalise, nameinlink]{cleveref}
\ytableausetup{nobaseline}
\newlength{\defbaselineskip}
\theoremstyle{plain}
\newtheorem{theorem}{Theorem}
\newtheorem{lemma}[theorem]{Lemma}

\theoremstyle{definition}

\newtheorem{example}[theorem]{Example}
\newtheorem{remark}[theorem]{Remark}

\newcommand{\boks}[1]{\ytableausetup{boxsize = #1 cm}}

\newcommand{\y}[1]{\ydiagram{#1}}
\newcommand{\yt}[1]{\ytableaushort{#1}}

\newcommand{\blk}{\underline{\hspace*{1em}}}
\newcommand{\hh}{\mathbf{h}}
\newcommand{\ee}{\mathbf{e}}
\newcommand{\rr}{\mathbf{r}}

\newcommand{\al}{\alpha}
\newcommand{\be}{\beta}
\newcommand{\ga}{\gamma}
\newcommand{\de}{\delta}
\newcommand{\eps}{\epsilon}
\newcommand{\la}{\lambda}
\newcommand{\si}{\sigma}
\newcommand{\smin}{\setminus}

\newcommand{\F}{\mathbb{F}}
\newcommand{\Z}{\mathbb{Z}}
\newcommand{\Q}{\mathbb{Q}}

\newcommand{\pow}{\boldsymbol{\psi}}
\newcommand{\mcER}{\mathcal{ER}}
\newcommand{\mcI}{\mathcal{I}}

\DeclareMathOperator{\ab}{ab}

\DeclareMathOperator{\cbt}{CBT}
\DeclareMathOperator{\cycC}{cycC}
\DeclareMathOperator{\cycP}{cycP}
\DeclareMathOperator{\dg}{dg}
\DeclareMathOperator{\rht}{RHT}
\DeclareMathOperator{\sgn}{sgn}
\DeclareMathOperator{\sort}{sort}
\DeclareMathOperator{\srht}{SRHT}
\DeclareMathOperator{\ssyt}{SSYT}

\DeclareMathOperator{\wt}{wt}
\DeclareMathOperator{\bt}{BT}
\DeclareMathOperator{\obt}{OBT}

\newcommand{\mcA}{\mathcal{A}}
\newcommand{\mcB}{\mathcal{B}}
\newcommand{\mcP}{\mathcal{P}}
\newcommand{\cc}{\mathbf{c}}
\newcommand{\bla}{\overline{\la}}
\newcommand{\bmu}{\overline{\mu}}
\newcommand{\tga}{\tilde{\gamma}}
\newcommand{\bB}{\overline{B}}
\DeclareMathOperator{\Surv}{Surv}
\DeclareMathOperator{\CS}{CS}

\let\bbmatrix\bordermatrix
\patchcmd{\bbmatrix}{8.75}{8.75}{}{}
\patchcmd{\bbmatrix}{\left(}{\left[}{}{}
\patchcmd{\bbmatrix}{\right)}{\right]}{}{}

\newlength{\cellsize}
\newcommand\tableau[1]{
\vcenter{
\let\\=\cr
\baselineskip=-16000pt
\lineskiplimit=16000pt
\lineskip=0pt
\halign{&\tableaucell{##}\cr#1\crcr}}}

\newcommand{\tableaucell}[1]{{%
\def \arg{#1}\def \void{}%
\ifx \void \arg
\vbox to \cellsize{\vfil \hrule width \cellsize height 0pt}%
\else
\unitlength=\cellsize
\begin{picture}(1,1)
\put(0,0){\makebox(1,1){$#1$}}
\put(0,0){\line(1,0){1}}
\put(0,1){\line(1,0){1}}
\put(0,0){\line(0,1){1}}
\put(1,0){\line(0,1){1}}
\end{picture}%
\fi}}

\begin{document}

\title{A local framework for proving combinatorial matrix inversion theorems}
 
\author{Aditya Khanna
\and Nicholas A. Loehr}
\thanks{This work was supported by a Gift from the Simons
 Foundation/SFARI (\#633564 to Nicholas Loehr).} 
\address{Dept. of Mathematics, Virginia Tech, Blacksburg, VA 24061-0123}
\email{adityakhanna@vt.edu, nloehr@vt.edu}
 
\begin{abstract}
Combinatorial transition matrices arise frequently in the theory of
symmetric functions and their generalizations. The entries of such matrices
often count signed, weighted combinatorial structures such as semistandard 
tableaux, rim-hook tableaux, or brick tabloids.  Bijective proofs that two 
such matrices are inverses of each other may be difficult to find. This paper
presents a general framework for proving such inversion results in the case
where the combinatorial objects are built up recursively by successively
adding some incremental structure such as a single horizontal strip or
rim-hook. In this setting, we show that a sequence of matrix inversion results 
$A_nB_n=I$ can be reduced to a certain ``local'' identity involving the 
incremental structures.  Here, $A_n$ and $B_n$ are matrices
that might be non-square, and the columns of $A_n$ and the rows of $B_n$
indexed by compositions of $n$. We illustrate the general theory with four
classical applications involving the Kostka matrices, the character tables of 
the symmetric group, incidence matrices for composition posets, and matrices
counting brick tabloids. We obtain a new, canonical bijective proof of an
inversion result for rectangular Kostka matrices, which complements the proof 
for the square case due to E\u{g}ecio\u{g}lu and Remmel. 
We also give a new bijective proof of the orthogonality result for the
irreducible $S_n$-characters that is shorter than
the original version due to White. 
\end{abstract}

\maketitle

\noindent\textbf{Keywords:} symmetric functions; quasisymmetric functions;
 transition matrices; Pieri Rules; semistandard tableaux; Kostka matrices; 
 rim-hook tableaux; brick tabloids; matrix inversion.
 
\noindent\textbf{2020 MSC Subject Classifications:}  
 05A19; 05A17; 15A09; 05E05. 

\section{Introduction}
\label{sec:intro} 

Combinatorial matrices occur ubiquitously in algebraic combinatorics,
especially in the theory of symmetric functions and their generalizations. 
Transition matrices connecting two bases of symmetric functions
(or quasisymmetric functions, etc.) very often have combinatorial
interpretations where the matrix entries count signed, weighted structures 
such as semistandard tableaux, rim-hook tableaux, or brick tabloids.  
For example, the classical Kostka matrix $\mathbf{K}$ gives the monomial 
expansion of the Schur symmetric functions. The matrix entry $K_{\la,\mu}$ 
identifies the coefficient of $m_{\mu}$ in $s_{\la}$ as the number of 
semistandard Young tableaux of shape $\la$ and content $\mu$.
E\u{g}ecio\u{g}lu and Remmel~\cite{inv-kostka} gave a combinatorial 
formula expressing each entry in $\mathbf{K}^{-1}$ as a signed sum
of special rim-hook tableaux. (See Section~\ref{sec:kostka} for detailed
definitions of the objects mentioned here.)

As another example, the character table of the symmetric group $S_n$
can be viewed as a matrix $\mathbf{X}=[\chi^{\la}_{\mu}]$ with rows and columns
indexed by integer partitions of $n$. Here, $\chi^{\la}_{\mu}$
is the value of the irreducible character $\chi^{\la}$ indexed by $\la$
on the conjugacy class of $S_n$ consisting of permutations of cycle type
$\mu$. Remarkably, $\chi^{\la}_{\mu}$ is also the coefficient of the
Schur function $s_{\la}$ when the power-sum symmetric function $p_{\mu}$
is written in the Schur basis. The entry $\chi^{\la}_{\mu}$ has a combinatorial
interpretation as the signed sum of rim-hook tableaux of shape $\la$ 
and content $\mu$. The inverse of $\mathbf{X}$ has entries 
$\chi^{\mu}_{\la}/z_{\la}$,
where $n!/z_{\la}$ is the number of permutations in $S_n$ with cycle
type $\la$.  The monograph~\cite{sagan} gives a very clear
exposition of the results stated in this paragraph. 

Bijective proofs that two combinatorial matrices are inverses of each other 
may be difficult to find. In the case of the Kostka matrix,
E\u{g}ecio\u{g}lu and Remmel~\cite{inv-kostka} gave an ingenious proof
based on a sign-reversing involution on pairs consisting of a semistandard
tableau and a special rim-hook tableau with compatible content.
In the case of the character matrix $\mathbf{X}$, the formula for
$\mathbf{X}^{-1}$ follows from the orthogonality relations for irreducible
characters. Finding a bijective proof based on rim-hook tableaux
is very challenging; Dennis White gave such a proof via an
intricate algorithmic construction~\cite{white-bij}. There are many
other instances of combinatorial transition matrices where inversion
results require elaborate algebraic manipulations, subtle bijective arguments,
or some combination of these.

Our goal here is to develop a general framework for proving 
inversion results for certain combinatorial matrices. 
We often have not just one matrix and its inverse, but a whole
family of matrices. For example, what we have called the Kostka
matrix is really a sequence of matrices $\mathbf{K}_n$ for $n\geq 0$,
where the rows and columns of $\mathbf{K}_n$ are
indexed by integer partitions of $n$. A simple recursion (based on removing
the largest value from a semistandard tableau) relates the
entries of $\mathbf{K}_n$ to entries in various smaller matrices
$\mathbf{K}_m$. 

Our general theory, presented in Section~\ref{sec:gen-framework},
considers two sequences of matrices $(A_n:n\geq 0)$ and $(B_n:n\geq 0)$
where the columns of $A_n$ and the rows of $B_n$
are indexed by compositions of $n$. These matrices may be rectangular
rather than square. We assume that the entries in $A_n$ can be computed
recursively from certain entries in $A_0,\ldots,A_{n-1}$, and similarly
for $B_n$. This setup applies to many combinatorial matrices (such as 
suitably adjusted versions of the Kostka matrices and their inverses)
where the combinatorial objects are built up recursively by 
successively adding some incremental structure such as a single horizontal 
strip, special rim-hook, etc.  In this setting, we show that the sequence of 
(one-sided) matrix inversion results $A_nB_n=I$ is equivalent to a certain 
``local'' identity involving the incremental structures. Intuitively,
this local identity contains the combinatorial essence of why each $B_n$
is a right-inverse of $A_n$.  Section~\ref{sec:automatic-bij} explains
how a bijective proof of the local identity can often be leveraged to create a
bijective proof of the matrix inversion result itself.

We illustrate our general theory with four classical applications 
that have relevance to transition matrices for bases of Sym
(the space of symmetric functions), QSym (the space of quasisymmetric 
functions), and NSym (the space of non-commutative symmetric functions).
Section~\ref{sec:kostka} studies rectangular versions of the Kostka
matrix and its inverse. Section~\ref{subsec:invert-Kostka} uses this
analysis to build a new bijective proof of the inversion result for 
rectangular Kostka matrices. This bijection is canonical 
(in a certain precise technical sense) and differs from the bijection
for the square case due to E\u{g}ecio\u{g}lu and Remmel~\cite{inv-kostka}.
Section~\ref{sec:RHT} studies rectangular versions of the matrix 
$\mathbf{X}$ and its inverse. Here, the local identity has an
``almost canonical'' bijective proof that can be expressed elegantly
in terms of abaci. Section~\ref{subsec:appl2-bij} lifts this to a
bijective proof of the matrix inversion result that is substantially
shorter than Dennis White's original version~\cite{white-bij}.
The proof uses novel combinatorial interpretations of $n!$ parametrized
by partitions $\la$ (see Remark~\ref{rem:novel-n!}).
Section~\ref{sec:comp-ref} uses our general theory to invert the
incidence matrices for posets of compositions of $n$ ordered by refinement.
Although this classical result is easy enough to prove algebraically,
the combinatorial proof of the local identity is still illuminating.
The latter proof also generalizes readily to prove a more daunting weighted 
version of the classical result (Section~\ref{subsec:wtd-appl3}) with relevance
to certain transition matrices for NSym.
Section~\ref{sec:brick-tabl} gives a fourth application involving
matrices counting brick tabloids. These matrices connect the power-sum
and monomial bases of symmetric functions.

Most of this paper (except for a few isolated remarks) can be read
without any prior knowledge of Sym, QSym, NSym, or the various combinatorial
structures mentioned in this introduction. The four application sections
are almost entirely independent of one another, except that
Lemma~\ref{lem:harmonic} and some definitions in Section~\ref{sec:comp-ref} 
are needed in Section~\ref{sec:brick-tabl}. 

\section{General Framework for Matrix Inversion}
\label{sec:gen-framework}

This section presents the general theory for proving combinatorial
matrix inversion results via reduction to local identities.
We begin by defining the needed notation and terminology.

\subsection{Notation and Definitions}
\label{subsec:defs}

For each integer $n\geq 0$, a \emph{composition of $n$} is a
list $\al=(\al_1,\al_2,\ldots,\al_s)$ of positive integers 
with $\al_1+\cdots+\al_s=n$. The \emph{size} of $\al$ is $|\al|=n$,
and the \emph{length} of $\al$ is $\ell(\al)=s$. 
The entries $\al_i$ are the \emph{parts} of $\al$.
Write $L(\al)=\al_s$ for the \emph{last part} of $\al$.
If $s\geq 1$, write $\al^*=(\al_1,\al_2,\ldots,\al_{s-1})$
for the \emph{truncation of $\al$}, which is the composition
of $n-\al_s$ obtained by deleting the last part of $\al$.
An \emph{integer partition} is a composition
$\la=(\la_1,\la_2,\ldots,\la_s)$ satisfying
$\la_1\geq\la_2\geq\cdots\geq\la_s$.

Write $C(n)$ for the set of compositions of $n$.
Write $P(n)$ for the set of partitions of $n$.
The empty sequence is the unique element in $C(0)$ and $P(0)$.
Write $[n]$ for the set $\{1,2,\ldots,n\}$.
Define $\sort:C(n)\rightarrow P(n)$ by letting $\sort(\al)$
be the weakly decreasing rearrangement of the list $\al\in C(n)$.
Given $\al\in\ C(n)$, the \emph{diagram} $\dg(\al)$ of $\al$ is the
set $\{(i,j)\in\Z^2: 1\leq i\leq \ell(\al), 1\leq j\leq \al_i\}$.
We visualize $\dg(\al)$ as a left-justified array of unit boxes
with $\al_i$ boxes in the $i$th row from the top. The box in row $i$,
column $j$ is identified with the pair $(i,j)\in\dg(\al)$.
\boks{0.2}
\ytableausetup{aligntableaux = center}
\begin{example}
    For $\al = (2,4,2,3)$ we have $\dg(\al) = \y{2,4,2,3}*[*(gray)]{0,2+1}$ and the marked box has the label $(2,3)$.
\end{example}
For any finite sets $R$ and $C$ and field $\F$, an \emph{$R\times C$ matrix}
is a function $A:R\times C\rightarrow\F$. For $r\in R$ and $c\in C$,
we call $A(r,c)$ the \emph{entry in row $r$, column $c$ of $A$}.
For any statement $Q$, let $\chi(Q)=1$ if $Q$ is true
and $\chi(Q)=0$ if $Q$ is false.
For any finite set $R$, let $I_R$ be the \emph{identity matrix}
with rows and columns indexed by $R$, 
which satisfies $I_R(\la,\mu)=\chi(\la=\mu)$ for all $\la,\mu\in R$.
We may omit the subscript $R$ if it is clear from context.

\subsection{Setup for the Matrix Inversion Framework}
\label{subsec:setup}

Next we describe the assumed setup for our matrix inversion framework.
The ingredients in the setup consist of two sequences of matrices,
with rows and columns indexed by certain combinatorial objects,
and two recursions that describe how the matrices in each sequence
are built from earlier matrices in that sequence.

We begin by describing the matrices.
Let $\F$ be a field (which in our applications will be $\Q$).
Suppose, for each integer $n\geq 0$, we have a finite set of objects $R(n)$,
where $R(0)$ consists of a single element. (In all applications considered
in this paper, $R(n)$ will be either $P(n)$ or $C(n)$.)
Suppose we are given a sequence of matrices $A_0,A_1,A_2,\ldots,A_n,\ldots$
where each $A_n$ is an $R(n)\times C(n)$ matrix with values in $\F$.
Thus, the rows of $A_n$ are indexed by objects in $R(n)$,
while the columns of $A_n$ are indexed by compositions of $n$.
We may abbreviate $A_n(\la,\be)$ as $A(\la,\be)$, since $n$ must be $|\be|$.
Also suppose we are given a sequence of matrices 
$B_0,B_1,B_2,\ldots,B_n,\ldots$, where each $B_n$ is a $C(n)\times R(n)$ matrix
with values in $\F$. Again we write $B(\be,\mu)=B_n(\be,\mu)$ where $n=|\be|$.

We continue by describing the recursions assumed as part of the setup.
For each $n>0$, $\la\in R(n)$, and $L\in [n]$, assume there is a
finite set $S(\la,L)\subseteq R(n-L)$ and an $\F$-valued weight function 
$\wt_A$ such that the matrices $A_n$ satisfy the recursion
\begin{equation}\label{eq:A-rec}
 A(\la,\be)=\sum_{\ga\in S(\la,L(\be))} \wt_A(\la,\ga)A(\ga,\be^*)
\end{equation}
and initial condition $A_0=[1]$. We interpret this recursion informally
as follows.  The entry $A(\la,\be)$ counts signed, weighted objects with
\emph{shape} $\la\in R(n)$ and \emph{content} 
$\be=(\be_1,\ldots,\be_s)=(\be^*,L(\be))$. The recursion asserts
that each such object can be built uniquely by starting with some
smaller shape $\ga\in R(n-L(\be))$, drawn from some set $S(\la,L(\be))$
of allowable shapes depending on $\la$ and $L(\be)$, choosing any object
of shape $\ga$ and content $\be^*$ counted by $A(\ga,\be^*)$, and
augmenting that object with some incremental structure. The effect
of the augmenting step is reflected by the factor $\wt_A(\la,\ga)$.
The initial condition amounts to the assumption that there is a
unique ``empty object'' with empty shape in $R(0)$ and empty content in $C(0)$.

We make the analogous assumption for the sequence of matrices $B_n$.
For each $n>0$, $\mu\in R(n)$, and $L\in [n]$, assume there is a
finite set $T(\mu,L)\subseteq R(n-L)$ and an $\F$-valued weight function 
$\wt_B$ such that the matrices $B_n$ satisfy the recursion
\begin{equation}\label{eq:B-rec}
 B(\be,\mu)=\sum_{\de\in T(\mu,L(\be))} \wt_B(\mu,\de)B(\be^*,\de)
\end{equation}
and initial condition $B_0=[1]$. This recursion has a similar
interpretation as before, but here the column index $\mu$ gives
the shape and the row index $\be$ gives the content of the
objects counted by the entries in $B_n$. In some applications,
the roles of shape and content in recursions~\eqref{eq:A-rec}
or~\eqref{eq:B-rec} may be interchanged.

\subsection{Matrix Inversion via the Local Identity}
\label{subsec:local-id}

We now come to the first main result of the general theory.

\begin{theorem}\label{thm:main}
Assume the setup in~\S\ref{subsec:setup}. The family of matrix identities
\begin{equation}\label{eq:gen-inv}
 A_n B_n = I_{R(n)}\quad\mbox{ for all $n\geq 0$} 
\end{equation}
is equivalent to the family of \emph{local identities}
\begin{equation}\label{eq:gen-local}
 \sum_{L=1}^n\sum_{\ga\in S(\la,L)\cap T(\mu,L)} \wt_A(\la,\ga)\wt_B(\mu,\ga)
  =\chi(\la=\mu)
 \quad\mbox{ for all $n>0$ and all $\la,\mu\in R(n)$.}
\end{equation}
\end{theorem}
\begin{proof}
For $n>0$ and $\la,\mu\in R(n)$, we compute the entry in row $\la$,
column $\mu$ of $A_nB_n$ as follows:
\begin{align*}
(A_n B_n)(\la,\mu)
 &= \sum_{\be\in C(n)} A(\la,\be)B(\be,\mu)
\\ &=\sum_{\be\in C(n)}
\left(\sum_{\ga\in S(\la,L(\be))} \wt_A(\la,\ga)A(\ga,\be^*)\right)\cdot
\left(\sum_{\de\in T(\mu,L(\be))} \wt_B(\mu,\de)B(\be^*,\de)\right).
\end{align*}
Each $\be\in C(n)$ has the form $\be=(\be^*,L)$ for
a unique $L=L(\be)\in [n]$ and a unique $\be^*\in C(n-L)$.
So we can rewrite the sum over $\be$ as a sum over $\be^*$ and $L$. We get:
\begin{align*}
(A_n B_n)(\la,\mu)
 &=\sum_{L=1}^n\sum_{\be^*\in C(n-L)}
\left(\sum_{\ga\in S(\la,L)} \wt_A(\la,\ga)A(\ga,\be^*)\right)\cdot
\left(\sum_{\de\in T(\mu,L)} \wt_B(\mu,\de)B(\be^*,\de)\right)
\\ &=\sum_{L=1}^n\sum_{\ga\in S(\la,L)}\sum_{\de\in T(\mu,L)}
  \wt_A(\la,\ga)\wt_B(\mu,\de)\sum_{\be^*\in C(n-L)} A(\ga,\be^*)B(\be^*,\de).
\end{align*}
We recognize the innermost sum as the definition of a matrix product entry, so
\begin{equation}\label{eq:gen-intmed}
(A_n B_n)(\la,\mu)=\sum_{L=1}^n\sum_{\ga\in S(\la,L)}\sum_{\de\in T(\mu,L)}
  \wt_A(\la,\ga)\wt_B(\mu,\de) (A_{n-L}B_{n-L})(\ga,\de).
\end{equation}

On one hand, assume~\eqref{eq:gen-inv} holds for all $n\geq 0$.
Then, in particular, $(A_nB_n)(\la,\mu)=\chi(\la=\mu)$
and $(A_{n-L}B_{n-L})(\ga,\de)=\chi(\ga=\de)$. Inserting these
expressions into~\eqref{eq:gen-intmed}, the double sum over
$\ga$ and $\de$ simplifies to a single sum over 
$\ga=\de\in S(\la,L)\cap T(\mu,L)$. 
So~\eqref{eq:gen-intmed} reduces to~\eqref{eq:gen-local}, as needed.

On the other hand, assume instead that~\eqref{eq:gen-local} holds.
We prove~\eqref{eq:gen-inv} by induction on $n$.
For $n=0$, the formula is true since $A_0=B_0=[1]$.
Fix $n>0$ and assume $A_mB_m=I_{R(m)}$ for all $m<n$.
Taking $m=n-L$ for $L=1,2,\ldots,n$, Equation~\eqref{eq:gen-intmed} becomes
\[ (A_n B_n)(\la,\mu)=\sum_{L=1}^n\sum_{\ga\in S(\la,L)}\sum_{\de\in T(\mu,L)}
  \wt_A(\la,\ga)\wt_B(\mu,\de)\chi(\ga=\de). \]
As before, the right side here simplifies to the left side
of the local identity~\eqref{eq:gen-local}. Thus,
$(A_nB_n)(\la,\mu)=\chi(\la=\mu)$ for all $\la,\mu\in R(n)$, 
proving that $A_nB_n=I_{R(n)}$.
\end{proof}

\begin{remark}\label{rem:sort-mat}
The theorem does not assert that $B_nA_n=I_{C(n)}$. 
In most of our applications, $R(n)$ will be $P(n)$, 
and hence $B_nA_n=I_{C(n)}$ cannot possibly hold
because the left side has rank at most $|P(n)|$.
However, we can often convert the one-sided matrix
inverse formula $A_nB_n=I_{R(n)}$ (for rectangular matrices $A_n,B_n$)
to a related formula involving square matrices.
In the square case, $AB=I$ automatically implies $BA=I$, as is well known.

In many applications, $A_n$ is a $P(n)\times C(n)$ matrix
satisfying the following \emph{sorting condition}: 
for all $\la\in P(n)$ and $\al,\be\in C(n)$,
if $\sort(\al)=\sort(\be)$ then $A(\la,\al)=A(\la,\be)$.
In this case, we can convert from rectangular to square matrices as follows.
Define $A_n':P(n)\times P(n)\rightarrow\F$ to be the restriction of
$A_n$ to the rows and columns indexed by partitions.
Define $B_n':P(n)\times P(n)\rightarrow\F$ by
\begin{equation}\label{eq:sort-inv}
 B_n'(\nu,\mu)=\sum_{\substack{\be\in C(n):\\\sort(\be)=\nu}} B_n(\be,\mu)
\quad\mbox{ for $\nu,\mu\in P(n)$.} 
\end{equation}
If we know $A_nB_n=I_{P(n)}$, then we can 
deduce $A_n'B_n'=I_{P(n)}$ as follows. For $\la,\mu\in P(n)$,
\begin{align*}
 (A_n'B_n')(\la,\mu) &= \sum_{\nu\in P(n)} A_n'(\la,\nu)B_n'(\nu,\mu)
 =\sum_{\nu\in P(n)} A_n(\la,\nu)
     \sum_{\substack{\be\in C(n):\\\sort(\be)=\nu}} B_n(\be,\mu)
\\ &= \sum_{\nu\in P(n)}\sum_{\substack{\be\in C(n):\\\sort(\be)=\nu}}
      A_n(\la,\be)B_n(\be,\mu)
 = \sum_{\be\in C(n)} A_n(\la,\be)B_n(\be,\mu)=(A_nB_n)(\la,\mu).
\end{align*}
\end{remark}

\begin{remark}
Here is an informal combinatorial interpretation of the 
local identities~\eqref{eq:gen-local}. Fix $\la,\mu\in R(n)$.
We consider all possible ways of transforming $\la$ into $\mu$
via the following two-step process. First, go from $\la$
to some intermediate object $\ga$ by removing some differential
structure of size $L$ (determined by the recursion for the $A$ matrix). 
Second, go from this $\ga$ to $\mu$ by adding another differential
structure of size $L$ (determined by the recursion for the $B$ matrix). 
For fixed $\ga$, the transition from $\la$ to $\mu$ via $\ga$
contributes a term $\wt_A(\la,\ga)\wt_B(\mu,\ga)$ that could
be negative. The sum of such terms over all feasible $\ga$
is required to be $1$ if $\la=\mu$ and $0$ if $\la\neq\mu$.
In later applications, we use the notation 
$G(\la,\mu)=\bigcup_{L=1}^n S(\la,L)\cap T(\mu,L)$
for the set of $\ga$ that are viable intermediate shapes
in the passage from $\la$ to $\mu$. The local identities~\eqref{eq:gen-local}
take the form $$\sum_{\ga\in G(\la,\mu)} \wt_A(\la,\ga)\wt_B(\mu,\ga)
 =\chi(\la=\mu).$$
\end{remark} 

\section{Application 1: Rectangular Kostka Matrices}
\label{sec:kostka}

As the first application of the general theory, we prove an inversion
result for a rectangular version of the Kostka matrices. In this case,
the local identities~\eqref{eq:gen-local} can be informally summarized 
by the slogan:
``A local inverse of a horizontal strip is a signed special rim-hook.''

\subsection{Semistandard Tableaux}
\label{subsec:SSYT}

Given $\la\in P(n)$ and $\be\in C(n)$, 
a \emph{semistandard Young tableau (SSYT)} 
of \emph{shape} $\la$ and \emph{content} $\be$ is a filling of the 
diagram of $\la$ using $\be_k$ copies of $k$ for each $k$, such that
values weakly increase reading left to right along each row, and
values strictly increase reading top to bottom along each column.
Let $\ssyt(\la,\be)$ be the set of all such SSYT.
In this first application of the general setup of~\S\ref{subsec:setup},
we let $A_n$ be the $P(n)\times C(n)$ \emph{rectangular Kostka matrix} 
with entries $A_n(\la,\be)=|\ssyt(\la,\be)|$.

\begin{example}
    Let $\be = (2,3,2)$ and $\la = (4,3)$. The entry $A_n(\la, \be) =2$ 
counts the following tableaux in $\ssyt(\la, \be)$:
    \boks{0.4}
    \[
    \yt{1122,233} \qquad \yt{1123,223}
    \]
\end{example}

A \emph{horizontal strip of size $L$} is a set of $L$ unit 
boxes occupying distinct columns. Given $\la\in P(n)$ and $\be\in C(n)$,
suppose $T\in\ssyt(\la,\be)$ and write
$\be=(\be_1,\ldots,\be_s)=(\be^*,L(\be))$.
For $1\leq i\leq s$, the cells containing value $i$ in $T$
form a horizontal strip of size $\be_i$, as is readily checked.
Also, the cells containing values $1,2,\ldots,i$ form the diagram
of some partition $\la^{(i)}$ where $\dg(\la^{(i)})\subseteq\dg(\la)$.
Let $S(\la,L)$ be the set of partitions $\ga\in P(n-L)$ such that
$\dg(\ga)\subseteq\dg(\la)$ and the set difference 
$\la/\ga=\dg(\la)\smin\dg(\ga)$ is a horizontal strip of size $L$.
There is exactly one way to construct any given $T\in\ssyt(\la,\be)$ 
as follows: choose $\ga\in S(\la,L(\be))$;
choose any $T^*\in\ssyt(\ga,\be^*)$; and fill the $L(\be)$ boxes
in $\dg(\la)\smin\dg(\ga)$ with the value $s=\ell(\be)$.
This proves the recursion
\begin{equation}\label{eq:A-rec1}
 A(\la,\be)=\sum_{\ga\in S(\la,L(\be))} A(\ga,\be^*),
\end{equation}
which is an instance of~\eqref{eq:A-rec} with $\wt_A(\la,\ga)=1$.
In fact, we have a bijection
$F:\ssyt(\la,\be)\rightarrow\bigcup_{\ga\in S(\la,L(\be))} \ssyt(\ga,\be^*)$,
where $F(T)$ is the SSYT obtained by removing the $L(\be)$ boxes 
containing the value $s=\ell(\be)$ from $T$. 
The construction prior to~\eqref{eq:A-rec1} describes how to compute $F^{-1}$.

\subsection{Special Rim-Hook Tableaux}
\label{subsec:SRHT}

E\u{g}ecio\u{g}lu and Remmel~\cite{inv-kostka} discovered a combinatorial
interpretation for the inverse of the (square) $P(n)\times P(n)$ Kostka 
matrix involving special rim-hook tableaux. 
We define a rectangular analogue of their inverse Kostka matrix,
which also appeared in~\cite{ELW}.
A \emph{special rim-hook} (SRH) of size $L$ is a set of $L$ unit boxes
that can be traversed by starting with a box in column $1$ (the leftmost
column) and successively taking one step to the right or up to go from one box
to the next.  The \emph{sign} of a SRH $\sigma$ occupying $r$ rows is
$\sgn(\sigma)=(-1)^{r-1}$.

Given $\mu\in P(n)$ and $\be\in C(n)$, 
a \emph{special rim-hook tableau (SRHT)} of \emph{shape} $\mu$
and \emph{content} $\be$ is a filling $S$ of the diagram of $\mu$ such
that for all $k$, the cells of $S$ containing $k$ form a special rim-hook
of size $\be_k$, and the values in column 1 weakly increase from top to bottom.
Let $\srht(\mu,\be)$ be the set of all such SRHT.
The \emph{sign} of an SRHT $S$ is the product of the signs of all special
rim-hooks in $S$.  Let $B_n$ be the $C(n)\times P(n)$ matrix defined by
$B_n(\be,\mu)=\sum_{S\in\srht(\mu,\be)} \sgn(S)$.

\begin{example}
For $\be = (3,2,4)$ and $\mu = (3,3,3)$, $\srht(\mu,\be)$ is the set
containing this object
\[ \yt{111,223,333} \]
which has sign $(-1)^{1-1}(-1)^{1-1}(-1)^{2-1} = -1$.
If $\be = (2,4,3)$ and $\mu = (3,3,3)$, we have this SRHT with sign $-1$:
\[ \yt{112,222,333} \]
For $\be = (4,2,3)$ and $\mu = (3,3,3)$, we see that $\srht(\mu, 
\be) = \varnothing$.
\end{example}

Let $S\in\srht(\mu,\be)$ where $\be=(\be_1,\ldots,\be_s)=(\be^*,L(\be))$.
Removing the special rim-hook $\sigma$ in $S$ of size $L(\be)$ containing the
value $s=\ell(\be)$ leaves a smaller SRHT $S^*$ of content $\be^*$
and some partition shape $\de$, as is readily checked.
Moreover, $\sgn(S)=\sgn(S^*)\sgn(\sigma)$.
Let $T(\mu,L)$ be the set of partitions $\de\in P(n-L)$ such that
$\dg(\de)\subseteq\dg(\mu)$ and $\mu/\de=\dg(\mu)\smin\dg(\de)$ is a
special rim-hook of size $L$. Removing the last special rim-hook gives a
bijection $G:\srht(\mu,\be)\rightarrow\bigcup_{\de\in T(\mu,L(\be))}
 \srht(\de,\be^*)$. Taking signs into account, we deduce the recursion
\begin{equation}\label{eq:B-rec1}
 B(\be,\mu)=\sum_{\de\in T(\mu,L(\be))} \sgn(\mu/\de)B(\be^*,\de),
\end{equation}
which is an instance of~\eqref{eq:B-rec} with $\wt_B(\mu,\de)=\sgn(\mu/\de)$.
A version of~\eqref{eq:B-rec1} appears in~\cite[Sec. 4]{ELW}.

\begin{remark}
For given $\mu\in P(n)$ and $L\in [n]$, there is at most
one way to remove a special rim-hook of size $L$ from the southeast rim of 
$\dg(\mu)$ to leave a partition shape. In other words, $|T(\mu,L)|\leq 1$.
Iterating this reasoning shows that $|\srht(\mu,\be)|\leq 1$
for all $\mu\in P(n)$ and all $\be\in C(n)$.
\end{remark}

\begin{example}
For $n = 4$, the rectangular matrices are as follows:
    \begin{footnotesize}
    \[ 
    A_4:\quad \bbmatrix{~ & 4 & 31 & 22 & 211 & 13 & 121 & 112 & 1111\cr
    4 & 1 & 1 & 1 & 1 & 1 & 1 & 1 & 1 \cr
31 & 0 & 1 & 1 & 2 & 1 & 2 & 2 & 3 \cr
22 & 0 & 0 & 1 & 1 & 0 & 1 & 1 & 2 \cr
211 & 0 & 0 & 0 & 1 & 0 & 1 & 1 & 3 \cr
1111 & 0 & 0 & 0 & 0 & 0 & 0 & 0 & 1
    }
    \qquad B_4:\quad \bbmatrix{~ & 4 & 31 & 22 & 211 & 1111\cr
4 & 1 & -1 & 0 & 1 & -1 \cr
31&0 & 1 & 0 & -1 & 1 \cr
22& 0 & 0 & 1 & -1 & 1 \cr
211 & 0 & 0 & 0 & 1 & -1 \cr
13 & 0 & 0 & -1 & 0 & 1 \cr
121 & 0 & 0 & 0 & 0 & -1 \cr
112 & 0 & 0 & 0 & 0 & -1 \cr
1111& 0 & 0 & 0 & 0 & 1} 
    \]
    \end{footnotesize}
\end{example}

\subsection{Proof of the Local Identity}
\label{subsec:prove-local1}

In this first application, the local identity~\eqref{eq:gen-local}
takes the following form.

\begin{theorem}\label{thm:local1}
Given $n>0$ and $\la,\mu\in P(n)$, let $G(\la,\mu)$ be the set of 
partitions $\ga$ such that $\dg(\ga)$ can be obtained either by removing
a nonempty horizontal strip from $\dg(\la)$ or by removing
a nonempty special rim-hook from $\dg(\mu)$. Then
\begin{equation}\label{eq:local1}
\sum_{\ga\in G(\la,\mu)} \sgn(\mu/\ga) =\chi(\la=\mu).
\end{equation}
Therefore (by Theorem~\ref{thm:main}) the rectangular Kostka matrix $A_n$
has a right-inverse $B_n$ whose entries count signed SRHT.
\end{theorem}
\begin{proof}
First consider the case $\la=\mu$. The set $G(\la,\mu)$ consists
of the single partition $\ga$ obtained by deleting the last part of $\la$.
This is because the only removable horizontal strip in $\dg(\la)$
that is also a removable special rim-hook of $\dg(\la)$ is the set
of all boxes in the lowest row of $\dg(\la)$. Since this set occupies
one row, $\sgn(\mu/\ga)=+1$ in this case, and~\eqref{eq:local1} holds.

Next consider the case $\la\neq\mu$. We accompany the general proof by a 
running example where $\lambda = (6,4,2,1)$ and $\mu = (4,3,3,3)$.
The diagrams are shown here:
\boks{0.3} 
\[\dg(\lambda) = \y{6,4,2,1} \quad \dg(\mu) = \y{4,3,3,3}\]
We prove that $G(\la,\mu)$ is either empty or consists of
exactly two elements where the corresponding special rim-hooks
$\mu/\ga$ have opposite signs. This suffices to prove~\eqref{eq:local1}.
Write $\mu=(\mu_1,\mu_2,\ldots,\mu_s)$. There are exactly $s$
special rim-hooks that can be removed from $\dg(\mu)$ to leave
another partition diagram $\ga$. Namely, for $i=1,2,\ldots,s$,
we can remove the SRH $\sigma_i$ that starts in the lowest box of
column $1$ and moves along the southeast rim of $\dg(\mu)$ until
it ends at the rightmost box in row $i$ of the diagram.
Let $\ga^i$ be the partition obtained by removing $\sigma_i$ from $\dg(\mu)$.
These are the only possible elements of $G(\la,\mu)$.
For our running example, the diagrams of $\ga^i$ are drawn below
using white boxes, with the corresponding SRH $\sigma_i$ depicted in gray.
\[ \dg(\ga^1)=  \y{2,2,2}*[*(lightgray)]{2+2,2+1,2+1,3} \quad 
   \dg(\ga^2)=  \y{4,2,2}*[*(lightgray)]{0,2+1,2+1,3}\quad
   \dg(\ga^3)=  \y{4,3,2}*[*(lightgray)]{0,0,2+1,3}\quad
   \dg(\ga^4)=  \y{4,3,3}*[*(lightgray)]{0,0,0,3} \]

\emph{Step~1.} We show $G(\la,\mu)$ cannot contain both $\ga^i$ and $\ga^j$
when $j>i+1$. For if this happened, consider the rightmost box $c$
in row $i+1$ of $\dg(\mu)$. In our running example, taking $i=2$ and $j=4$,
the box $c$ is marked in black below:
\[\y{4,3,3,3}*[*(black)]{0,0,2+1} \]
Box $c$ and the box just above $c$ both
belong to $\sigma_i$ and so do not belong to $\dg(\ga^i)$. Since
$\la/\ga^i$ is a horizontal strip, $c$ cannot belong to $\dg(\la)$.
On the other hand, because $j>i+1$, $c$ does not belong to $\sigma_j$
and so does belong to $\dg(\ga^j)$, which is contained in $\dg(\la)$.
Thus $c$ does belong to $\dg(\la)$, which gives a contradiction.

\emph{Step~2.} 
Suppose $G(\la,\mu)$ is nonempty, say with $\ga^i\in G(\la,\mu)$.
We prove $G(\la,\mu)=\{\ga^i,\ga^{i+1}\}$ 
      or $G(\la,\mu)=\{\ga^i,\ga^{i-1}\}$.
Let $\eta_i=\la/\ga^i$, which must be a horizontal strip.
We can pass from $\dg(\mu)$ to $\dg(\la)$ by first removing $\sigma_i$ to 
produce $\dg(\ga^i)$, then adding $\eta_i$ to reach $\dg(\la)$.
Let $c$ be the rightmost box in row $i$ of $\dg(\mu)$,
which is in $\sigma_i$ and thus not in $\dg(\ga^i)$.
Let $R$ be the set of boxes of $\sigma_i$ lying in row $i$ of $\dg(\mu)$.
In our running example, we may take $i=2$. The black box in the next
figure is $c$, the rest of $\sigma_i$ is shown in gray, and in this
instance, $R=\{c\}$.
\[\dg(\mu) = \y{4,3,3,3}*[*(lightgray)]{0,0,2+1,3}*[*(black)]{0,2+1} 
\xrightarrow{-\sigma_2} \dg(\ga^2) = \y{4,2,2} \xrightarrow{+\eta_2} 
\dg(\la)=\y{6,4,2,1}*[*(lightgray)]{4+2,3+1,0,1}*[*(black)]{0,2+1} \]

Case~1: Assume $c$ is in $\dg(\la)$. 
Then we must have $c\in\eta_i$, hence $R\subseteq\eta_i\subseteq\dg(\la)$.
\begin{itemize}
\item Case~1a: Assume $i<s$, so $\ga^{i+1}$ is defined. 
In this situation, $\sigma_{i+1}=\sigma_i\smin R$,
$\ga^{i+1}=\ga^i\cup R$, $\dg(\ga^{i+1})\subseteq\dg(\la)$,
and (critically) $\la/\ga^{i+1}=\eta_i\smin R$ is a horizontal strip.
Thus $\ga^{i+1}\in G(\la,\mu)$, and Step~1 applies to show
$G(\la,\mu)=\{\ga^i,\ga^{i+1}\}$.  Note $\sgn(\sigma_{i+1})=-\sgn(\sigma_i)$.
In our running example, the next figure shows how to go from $\mu$ to $\la$
by removing $\sigma_3$ and then adding $\eta_2\setminus R$.
\[\dg(\mu) = \y{4,3,3,3}*[*(lightgray)]{0,0,2+1,3} \xrightarrow{-\sigma_3} 
\dg(\ga^3) = \y{4,3,2} \xrightarrow{+\eta_2\setminus R} 
\dg(\la)=\y{6,4,2,1}*[*(lightgray)]{4+2,3+1,0,1}\]
\item Case~1b: Assume $i=s$; we rule out this case by deducing a contradiction.
Here, $\sigma_s$ and $R$ must both be the set of all cells in the lowest row 
of $\dg(\mu)$. But since $R$ is a subset of $\eta_s$, where $\eta_s$ has 
the same size as $\sigma_s$, we must have $\eta_s=R$. This in turn
forces $\la=\mu$, contradicting the assumption $\la\neq\mu$.
\end{itemize}

Case~2: Assume $c$ is not in $\dg(\la)$.
\begin{itemize}
\item Case~2a: Assume $i>1$, so $\ga^{i-1}$ is defined.
Let $R'$ be the set of cells of $\sigma_{i-1}$ lying in row $i-1$ of $\dg(\mu)$.
By assumption on $c$, the horizontal strip $\eta_i$ 
cannot contain any cell of row $i$
directly below a cell of $R'$ (although $\eta_i$ might contain some cells
in row $i-1$ to the right of $R'$). We now see that
$\sigma_{i-1}=\sigma_i\cup R'$, $\ga^{i-1}=\ga^i\smin R'$,
$\dg(\ga^{i-1})\subseteq\dg(\ga^i)\subseteq\dg(\la)$, 
and (critically) $\la/\ga^{i-1}=\eta_i\cup R'$ is a horizontal strip.
Thus $\ga^{i-1}\in G(\la,\mu)$, and Step~1 applies to show
$G(\la,\mu)=\{\ga^i,\ga^{i-1}\}$.  Note $\sgn(\sigma_{i-1})=-\sgn(\sigma_i)$. 

We illustrate this case with a new example. Let $\mu = (4,3,2)$ and 
$\la = (7,2)$. We label the special rim-hooks and horizontal strips in gray
and label the cell $c = (2,3)$ in black. On one hand, we remove $\si_2$ 
from $\dg(\mu)$ to reach the diagram of $\ga^2=(4,1)$, 
then add $\eta_2$ to $\dg(\ga^2)$ to obtain $\dg(\la)$.  On the other hand, 
we remove $\si_1 = \si_2 \cup R'$ from $\dg(\mu)$ to reach 
the diagram of $\ga^1=(2,1)$, then add $\eta_1$ to $\dg(\ga^1)$ 
to obtain $\dg(\la)$.
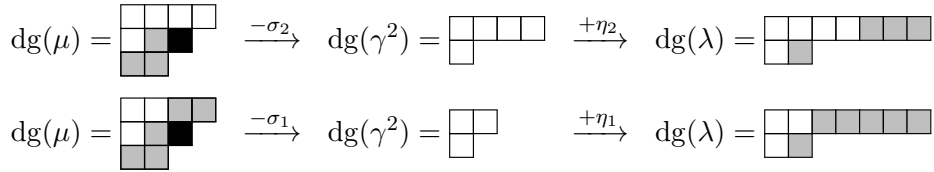
\[ \begin{array}{lclcl}
\dg(\mu)= \y{4,3,2}*[*(lightgray)]{0,1+1,2}*[*(black)]{0,2+1} 
& \xrightarrow{-\si_2} 
& \dg(\ga^2) = \y{4,1} 
& \xrightarrow{+\eta_2} 
& \dg(\la)=\y{4,1}*[*(lightgray)]{4+3,1+1} \\[15pt]
\dg(\mu) = \y{4,3,2}*[*(lightgray)]{2+2,1+1,2}*[*(black)]{0,2+1} 
& \xrightarrow{-\si_1} 
& \dg(\ga^2) = \y{2,1} 
& \xrightarrow{+\eta_1} 
& \dg(\la)=\y{2,1}*[*(lightgray)]{2+5,1+1}
\end{array} \]

\item Case~2b: Assume $i=1$; we rule out this case by deducing a contradiction.
The special rim-hook $\sigma_1$ starts in column $1$ and ends at cell $c$
in row $1$ and column $\mu_1$, so $\sigma_1$ contains at least $\mu_1$ boxes.
So the horizontal strip $\eta_1$, which has the same size as $\sigma_1$,
must have at least $\mu_1$ boxes, all of which occupy at least $\mu_1$ columns. Since $\eta_1$ is added to 
$\dg(\ga^1)\subseteq\dg(\mu)$, cell $c$ must belong to 
$\eta_1\subseteq\dg(\la)$, contrary to the assumption of Case~2. \qedhere
\end{itemize}
\end{proof}

\begin{remark}
Bender and Knuth~\cite{bender-knuth} gave a bijective proof that
the rectangular Kostka matrix satisfies the sorting condition of
Remark~\ref{rem:sort-mat}; see also~\cite[Thm. 9.27]{loehr-comb}.
By Remark~\ref{rem:sort-mat}, we can deduce E\u{g}ecio\u{g}lu and Remmel's 
inversion result~\cite[Thm. 1]{inv-kostka} for the square $P(n)\times P(n)$ 
Kostka matrix from our rectangular version. For $S\in\srht(\mu,\be)$, 
those authors refer to $\nu=\sort(\be)$ as the \emph{type} of $S$.
Their square inverse Kostka matrix has entries
$K^{-1}(\nu,\mu)=\sum_{\be:\,\sort(\be)=\nu}\sum_{S\in\srht(\mu,\be)} \sgn(S)$,
in accordance with~\eqref{eq:sort-inv}.
\end{remark}

\begin{remark}
By the proof of Theorem~\ref{thm:local1}, there is a bijection
$\ga^i\mapsto \ga^{i\pm 1}$
that matches the two oppositely-signed objects contributing to the sum,
in the case where $\la\neq\mu$ and the sum is not already empty.
Since there is only one positive object and one negative object,
this is a \emph{canonical} bijection (not relying on arbitrary choices).
We build this up to a canonical bijective proof of $A_nB_n=I_{P(n)}$
in Section~\ref{subsec:invert-Kostka}.
\end{remark}

\section{Application 2: Rim-Hook Tableaux}
\label{sec:RHT}

As the second application of the general theory, we prove an inversion
result for a rectangular version of the transition matrices between
power-sum symmetric functions and Schur functions. 
In this case, the local identities~\eqref{eq:gen-local} can be informally 
summarized by the slogan: ``A local inverse of a signed rim-hook is a 
rescaled signed rim-hook.'' 

\subsection{Rim-Hook Tableaux}
\label{subsec:RHT}

A \emph{rim-hook} of size $L$ is a set of $L$ unit boxes that can be
traversed by starting at some box and successively moving one step to the right or one step up
to go from one box to the next. The \emph{sign}
of a rim-hook $\sigma$ occupying $r$ rows is $\sgn(\sigma)=(-1)^{r-1}$.
Given $\la\in P(n)$ and $\be\in C(n)$, a \emph{rim-hook tableau (RHT)}
of \emph{shape} $\la$ and \emph{content} $\be$ is a filling $S$ of
the diagram of $\la$ such that for all $k$, the cells of $S$ containing $k$
form a rim-hook of size $\be_k$, and the cells of $S$ containing $1,2,\ldots,k$
form a partition diagram.  Let $\rht(\la,\be)$ be the set of all such RHT.
The \emph{sign} of an RHT $S$ is the product of the signs of all rim-hooks 
in $S$. In this second application of the general setup
of~\S\ref{subsec:setup}, we let $A_n$ be the $P(n)\times C(n)$ matrix
with entries $A_n(\la,\be)=\sum_{S\in\rht(\la,\be)} \sgn(S)$.
The symmetric function literature often uses
$\chi^{\la}_{\be}$ to denote $A_n(\la,\be)$.
\boks{0.4}
\begin{example}
Let $\lambda = (4,3,3,1)$ and $\beta = (3,4,4)$. 
We find $A(\lambda, \beta) = -2$ based on the two RHT shown below,
which both have sign $-1$.
\[ \yt{1133,123,223,2} \qquad \yt{1122,122,333,3} \]
\end{example}

Let $U\in\rht(\la,\be)$ where $\be=(\be_1,\ldots,\be_s)=(\be^*,L(\be))$.
Removing the rim-hook $\sigma$ in $U$ of size $L(\be)$ containing the
value $s=\ell(\be)$ leaves a smaller RHT $U^*$ of content $\be^*$ and
some shape $\ga$, by definition. Moreover,
$\sgn(U)=\sgn(U^*)\sgn(\sigma)$. Let $S(\la,L)$ be the set of 
partitions $\ga\in P(n-L)$ such that $\dg(\ga)\subseteq\dg(\la)$
and $\la/\ga=\dg(\la)\smin\dg(\ga)$ is a rim-hook of size $L$. Removing
the last rim-hook gives a bijection $G:\rht(\la,\be)\rightarrow
\bigcup_{\ga\in S(\la,L(\be))} \rht(\ga,\be^*)$. Taking signs into
account, we deduce the recursion
\begin{equation}\label{eq:A-rec2}
 A(\la,\be)=\sum_{\ga\in S(\la,L(\be))} \sgn(\la/\ga)A(\ga,\be^*),
\end{equation}
which is an instance of~\eqref{eq:A-rec} with $\wt_A(\la,\ga)=\sgn(\la/\ga)$.

\subsection{The Scaling Factors $Z_{\be}$.}
\label{subsec:scale_Z_beta}

Given a composition $\be=(\be_1,\be_2,\ldots,\be_s)\in C(n)$, 
define the integer 
\begin{equation}\label{eq:Zbe}
 Z_{\be}=\be_1(\be_1+\be_2)(\be_1+\be_2+\be_3)\cdots
  (\be_1+\be_2+\cdots+\be_s).
\end{equation}
For instance, $Z_{(3,4,4)} = 3\cdot 7 \cdot 11 = 231$ 
but $Z_{(4,3,4)} = 4\cdot 7 \cdot 11 = 308$. 
Since the last factor is $|\be|=n$, we have $Z_{\be}=n\cdot Z_{\be^*}$.
Let $B_n$ be the $C(n)\times P(n)$ matrix defined by
$B_n(\be,\mu)=\sum_{S\in\rht(\mu,\be)} Z_{\be}^{-1}\sgn(S)$
for $\be\in C(n)$ and $\mu\in P(n)$. Observe that $B_n$ is 
obtained from $A_n$ by rescaling all entries in each column $\be$ 
by $Z_{\be}^{-1}$ and then transposing the matrix.

Define $T(\mu,L)$ to be the set of partitions $\de\in P(n-L)$ such that 
$\dg(\de)\subseteq\dg(\mu)$ and $\mu/\de=\dg(\mu)\smin\dg(\de)$ 
is a rim-hook of size $L$ (so $T(\mu,L)$ is the same as $S(\mu,L)$
in this application).  Removing the last rim-hook $\sigma$ of 
$U\in\rht(\mu,\be)$ produces $U^*\in\rht(\de,\be^*)$ for a unique 
$\de\in T(\mu,L)$.
Note that $Z_{\be}^{-1}\sgn(U)=n^{-1}Z_{\be^*}^{-1}\sgn(U^*)\sgn(\sigma)$
where $n=|\be|=|\mu|$ and $\sigma=\mu/\de$. So we get the recursion
\begin{equation}\label{eq:B-rec2}
 B(\be,\mu)=\sum_{\de\in T(\mu,L(\be))} |\mu|^{-1}\sgn(\mu/\de)B(\be^*,\de),
\end{equation}
which is an instance of~\eqref{eq:B-rec} with 
$\wt_B(\mu,\de)=|\mu|^{-1}\sgn(\mu/\de)$.

\begin{example}
For $n = 4$, the rectangular matrices are as follows:
\begin{footnotesize}
\[ 
A_4:\ \bbmatrix{~ & 4 & 31 & 22 & 211 & 13 & 121 & 112 & 1111\cr
 4 & 1 & 1 & 1 & 1 & 1 & 1 & 1 & 1 \cr
31 & -1 & 0 & -1 & 1 & 0 & 1 & 1 & 3 \cr
22 & 0 & -1 & 2 & 0 & -1 & 0 & 0 & 2 \cr
211 & 1 & 0 & -1 & -1 & 0 & -1 & -1 & 3 \cr
1111 & -1 & 1 & 1 & -1 & 1 & -1 & -1 & 1 } \quad
B_4:\ \bbmatrix{~ & 4 & 31 & 22 & 211 & 1111\cr
4 & 1/4 & -1/4 & 0 & 1/4 & -1/4 \cr
31& 1/12 & 0 & -1/12 & 0 & 1/12 \cr
22& 1/8 & -1/8 & 2/8 & -1/8 & 1/8 \cr
211 & 1/24 & 1/24 & 0 & -1/24 & -1/24 \cr
13 & 1/4 & 0 & -1/4 & 0 & 1/4 \cr
121 & 1/12 & 1/12 & 0 & -1/12 & -1/12 \cr
112 & 1/8 & 1/8 & 0 & -1/8 & -1/8 \cr
1111& 1/24 & 3/24 & 2/24 & 3/24 & 1/24 }
\]
\end{footnotesize}
\end{example}

\subsection{Abacus Notation for Partitions} 
\label{subsec:abacus-def}

It is often useful to represent integer partitions by abaci,
especially when performing operations that affect the southeast rim
of the partition diagram. Given $\la\in P(n)$ and any $N\geq\ell(\la)$,
the \emph{$N$-bead abacus representing $\la$} is the binary
word $\ab^N(\la)=w_1w_2w_3\cdots$ defined as follows. The word begins
with $N-\ell(\la)$ copies of $1$ (each $1$ represents a bead on an
abacus). To continue building the word, move along the southeast rim
of $\dg(\la)$, from the southwest corner to the northeast corner,
by a succession of unit-length east steps and north steps. Record a $0$
(gap) in the abacus word for each east step, and record a $1$ (bead)
in the abacus word for each north step. The word terminates with
an infinite sequence of $0$s (gaps). The value of $N$ is usually irrelevant
as long as it is large enough to accommodate any operations performed
on the partition and its associated abacus.

For instance, consider the operation of removing a rim-hook of size $L$
from the rim of $\dg(\la)$. It is not hard to check that this operation
corresponds to making a bead ``jump'' from some position $i$ on $\ab^N(\la)$
to a position $i-L$ that contains a gap. In other words, removal of the
rim-hook $\sigma$ of size $L$ modifies the abacus by changing some $w_i$ from 
$1$ to $0$ and $w_{i-L}$ from $0$ to $1$. Furthermore, if there are $b$ beads
in positions $i-L+1,\ldots,i-1$ on $\ab^N(\la)$, then $\sgn(\sigma)=(-1)^b$.
(For a more detailed verification, see~\cite[Sec. 3.3]{loehr-abacus}
 or~\cite[Thm. 10.13]{loehr-comb}.)
On the other hand, adding a rim-hook of size $L$ to $\dg(\la)$ is
accomplished on the abacus by making a bead jump from position $i$
to some position $i+L$ that contains a gap. Here, the number of beads
on the abacus must satisfy $N\geq \ell(\la)+L$ to ensure that there 
are enough beads to accommodate all possible rim-hooks that might be added.

\begin{example}
Starting with the partition $\la=(4,3,3,2,2,1)$, we can remove a rim-hook
of size $L=5$ to obtain $\mu=(4,2,1,1,1,1)$ as shown in the following
partition diagrams:
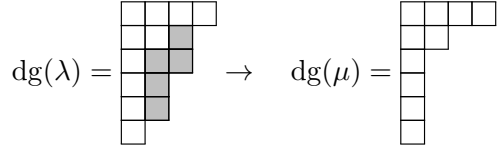
\boks{0.3}
\[ \dg(\la)=\y{4,3,3,2,2,1}*[*(lightgray)]{0,2+1,1+2,1+1,1+1} \to\quad
   \dg(\mu)=\y{4,2,1,1,1,1} \]
The corresponding operation on the $9$-bead abaci representing $\mu$
and $\nu$ is shown here:

\begin{center}
\begin{tikzpicture}
    \node (formula) [] 
{$\ab^9(\la)=11101\underline{0}1101\underline{1}01000\ldots \quad\to\quad
  \ab^9(\mu)=11101\underline{1}1101\underline{0}01000\ldots $};
    \draw[-latex,black] ($(formula.north west)+(+3.7,-0.2)$) arc
    [
        start angle=20,
        end angle=160,
        x radius=0.5cm,
        y radius =0.5cm
    ] ;
\end{tikzpicture}
\end{center}
The sign associated with this rim-hook removal is $(-1)^3 = -1$.
\end{example}

\subsection{Proof of the Local Identity}
\label{subsec:prove-local2}

In this second application, the local identity~\eqref{eq:gen-local}
takes the following form.

\begin{theorem}\label{thm:local2}
Given $n>0$ and $\la,\mu\in P(n)$, let $G(\la,\mu)$ be the set of 
partitions $\ga$ that can be obtained either by removing
a nonempty rim-hook from $\la$ or by removing
a nonempty rim-hook from $\mu$. Then
\begin{equation}\label{eq:local2}
\sum_{\ga\in G(\la,\mu)} n^{-1}\sgn(\la/\ga)\sgn(\mu/\ga) =\chi(\la=\mu).
\end{equation}
Therefore (by Theorem~\ref{thm:main}) the matrix $A_n$ whose entries count
signed RHT has a right-inverse $B_n$, a rescaled version of the transpose 
of $A_n$.
\end{theorem}
\begin{proof}
First consider the case $\la=\mu$. The left side of~\eqref{eq:local2}
becomes $\sum_{\ga\in G(\la,\la)} n^{-1}\sgn(\la/\ga)^2=|G(\la,\la)|/n$,
so it suffices to show that $|G(\la,\la)|=n=|\la|$. This amounts to showing
that there are exactly $n$ nonempty rim-hooks that can be removed from
the border of $\dg(\la)$. We give a bijection from $\dg(\la)$ 
(a set of $n$ unit boxes) onto this set of rim-hooks.
Given a box $c\in\dg(\la)$, move down (due south) from $c$ to reach
a box $c'$ on the southern border of $\dg(\la)$, and move right (due east)
from $c$ to reach a box $c''$ on the eastern border of $\dg(\la)$.
The bijection maps $c$ to the rim-hook that moves northeast along the
border starting at $c'$ and ending at $c''$. The inverse bijection takes
a rim-hook on the border of $\dg(\la)$, say going northeast from
a box $c_1$ to a box $c_2$, and maps this rim-hook to the box of $\dg(\la)$
in the same row as $c_2$ and the same column as $c_1$. In the diagram 
of $\la=(5,5,4,4,3)$ shown below, we label in gray the rim-hook 
corresponding to the black cell $c$.

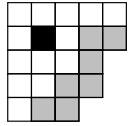
\[ \y{5,5,4,4,3}*[*(black)]{0,1+1}*[*(lightgray)]{0,3+2, 3+1, 2+2,1+2} \]

To handle the case $\la\neq\mu$, it suffices to prove
\[ \sum_{\ga\in G(\la,\mu)} \sgn(\la/\ga)\sgn(\mu/\ga)=0. \]
We do this by showing that $G(\la,\mu)$ is either empty or contains
exactly two objects contributing oppositely-signed terms to the sum.

Consider the $N$-bead abaci $\ab^N(\la)$ and $\ab^N(\mu)$,
 where $N=\max(\ell(\la),\ell(\mu))$.
Assuming $G(\la,\mu)\neq\emptyset$, there must be a way go from
$\dg(\la)$ to $\dg(\mu)$ by removing some rim-hook of size $L$
from $\dg(\la)$ to get $\dg(\ga)$ for some $\ga\in G(\la,\mu)$, then
adding some different rim-hook of size $L$ to $\dg(\ga)$ to get $\dg(\mu)$.
Translating to abaci, there must be a way to go from
$\ab^N(\la)$ to $\ab^N(\mu)$ by the following two-step process.
First, a bead in $\ab^N(\la)$ jumps from some position $i$ to
a gap in some position $i-L$, producing $\ab^N(\ga)$ for some $\ga$.
Second, a different bead in $\ab^N(\ga)$ jumps from some position $j\neq i$ 
to a gap in position $j+L$, producing $\ab^N(\mu)$.
We must have $j\neq i-L$ since $\la\neq\mu$.

The key observation is that there is exactly one other way to
execute this two-step process (with new choices of $i,j,L,\ga$)
to convert $\ab^N(\la)$ to $\ab^N(\mu)$. In more detail, first note that:
\begin{itemize}
\item positions $i$ and $j$ contain beads in $\ab^N(\la)$ 
and gaps in $\ab^N(\mu)$; 
\item positions $i-L$ and $j+L$ contain gaps in $\ab^N(\la)$ 
and beads in $\ab^N(\mu)$;
\item each other position contains the same thing (bead or gap)
 in $\ab^N(\la)$ and in $\ab^N(\mu)$.
\end{itemize}

Consider the case $i<j$, so $i-L<i<j<j+L$.  
The only other way to go from $\ab^N(\la)$
to $\ab^N(\mu)$ is to move a bead from position $j$ to position $i-L$
(jumping down by $L'=j-i+L>0$ positions) and then move a bead from
position $i$ to position $j+L$ (jumping up $L'$ positions). 
To compare the signs of the two ways, let $\ab^N(\la)$ have
$a$ beads strictly between positions $i-L$ and $i$,
$b$ beads strictly between positions $i$ and $j$, and
$c$ beads strictly between positions $j$ and $j+L$.
For the first bead motion, 
the bead jump from $i$ to $i-L$ contributes $(-1)^a$, and
the bead jump from $j$ to $j+L$ contributes $(-1)^c$.
For the second bead motion,
the bead jump from $j$ to $i-L$ contributes $(-1)^{a+b+1}$
(noting that this bead jumps over a bead at position $i$),
while the bead jump from $i$ to $j+L$ contributes $(-1)^{b+c}$
(noting there is no longer a bead at position $j$).
Since $(-1)^{a+2b+c+1}=-(-1)^{a+c}$, the two terms have opposite signs. 

Consider the case $j<i$ and $i-j<L$. Then $i-L<j<i<j+L$.
As in the previous case, the other way to go from $\ab^N(\la)$
to $\ab^N(\mu)$ is to move the bead in position $j$ to position $i-L$,
then move the bead in position $i$ to position $j+L$. A sign analysis
(similar to the first case) shows that the two ways lead to terms
with opposite signs.

Consider the case $j<i$ and $i-j\geq L$.
We cannot have $i-j=L$ since $j\neq i-L$ as noted earlier.
We cannot have $i-L=j+L$ since otherwise we could not move
a bead from $j$ to $j+L$ on $\ab^N(\ga)$. 
The possible orderings in this case are therefore
$j<j+L<i-L<i$ or $j<i-L<j+L<i$.
In both orderings, the second way to go from $\ab^N(\la)$
to $\ab^N(\mu)$ first moves a bead from $i$ to $j+L$,
then moves a bead from $j$ to $i-L$. As before, we can check
that the two ways lead to terms of opposite signs.
\end{proof}

\begin{remark}
When $\la\neq\mu$ and the sum on the left side of~\eqref{eq:local2}
is not vacuous, the proof just given shows there is exactly one
positive object and one negative object contributing to this sum.
Thus we have a \emph{canonical} bijection (not relying on arbitrary
choices) proving this identity. The bijection when $\la=\mu$,
while not canonical in this technical sense, is still quite natural.
We build this up to an ``almost canonical'' bijective proof of 
$A_n(n!B_n)=n!I_{P(n)}$ in Section~\ref{subsec:appl2-bij}.
\end{remark}

\begin{example}\label{ex:local2}
Let $\lambda = (9,8,6,6,5,4,4,2)$ and $\mu = (9,9,9,7,5,3,1,1)$, so $N=8$.
The first way to go from $\ab^N(\lambda)$ to $\ab^N(\mu)$ takes $L=5$ and moves 
a bead from position $i = 6$ to $i - L = 1$, then moves a bead from 
position $j = 10$ to $j+L = 15$, as shown by the upper arrows in the diagram
below. The second way takes $L= 10 - 6 + 5 = 9$ instead and moves a bead from position $10$ 
to $1$, then moves a bead from position $6$ to $15$, as shown by the lower
arrows in the diagram. The signs are $(-1)^{2+2} =+1$ for the first way
and $(-1)^{4 + 3} = -1$ for the second way.
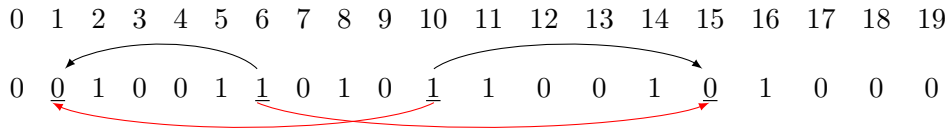
\begin{center}
\begin{tikzpicture}
    \node (formula) [] {$\begin{tabular}{*{20}{c}}
  0& 1 & 2 & 3 & 4 & 5 & 6 & 7 & 8 & 9 & 10 & 11 & 12 & 13 & 14 & 15 & 16 & 17 & 18 & 19 
\\ & &  &  &  &  &  &  &  &  &  &  &  &  &  &  &  &  &  &  \\
 0 & \underline{0} & 1 & 0 & 0 & 1 & \underline{1} & 0 & 1 & 0 & 
\underline{1} & 1 & 0 & 0 & 1 & \underline{0} & 1 & 0 & 0 & 0
\end{tabular} $};
    \draw[-latex,black] ($(formula.north west)+(+3.6,-1.0)$) arc
    [
        start angle=20,
        end angle=160,
        x radius=1.36cm,
        y radius =0.5cm
    ] ;
        \draw[-latex,black] ($(formula.north west)+(5.95,-1.0)$) arc
    [
        start angle=160,
        end angle=20,
        x radius=1.9cm,
        y radius =0.5cm
    ] ;
            \draw[-latex,red] ($(formula.north west)+(3.6,-1.45)$) arc
    [
        start angle=200,
        end angle=340,
        x radius=3.2cm,
        y radius =0.5cm
    ] ;
    \draw[-latex,red] ($(formula.north west)+(5.95,-1.45)$) arc
    [
        start angle=340,
        end angle=200,
        x radius=2.7cm,
        y radius =0.5cm
    ] ;
\end{tikzpicture}
\end{center}
Translating back to partition diagrams, let
$\ga=(9,8,6,6,5,3,1,1)$ and $\tilde{\ga}=(9,8,6,4,3,3,1,1)$.
The first way to go from $\la$ to $\mu$ removes a rim-hook
$R=\la/\ga$ from $\dg(\la)$ to reach $\dg(\ga)$, then
adds a rim-hook $S=\mu/\ga$ to $\dg(\ga)$ to reach $\dg(\mu)$, as shown here:
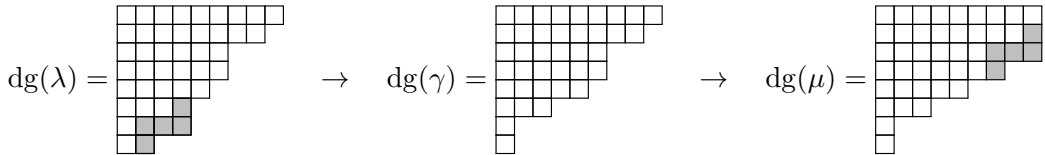
\boks{0.23}
\[ \dg(\lambda)  = \y{9,8,6,6,5,4,4,2}*[*(lightgray)]{0,0,0,0,0,3+1,1+3,1+1} 
\quad\to\quad \dg(\gamma) = \y{9,8,6,6,5,3,1,1} 
\quad\to\quad \dg(\mu)= \y{9,8,6,6,5,3,1,1}*[*(lightgray)]{0,8+1,6+3,6+1}\] 
The second way to go from $\la$ to $\mu$ removes a rim-hook
$\tilde{R}=\la/\tilde{\ga}$ from $\dg(\la)$ to reach $\dg(\tilde{\ga})$, then
adds a rim-hook $\tilde{S}=\mu/\tilde{\ga}$ to $\dg(\tilde{\ga})$ 
to reach $\dg(\mu)$, as shown here:
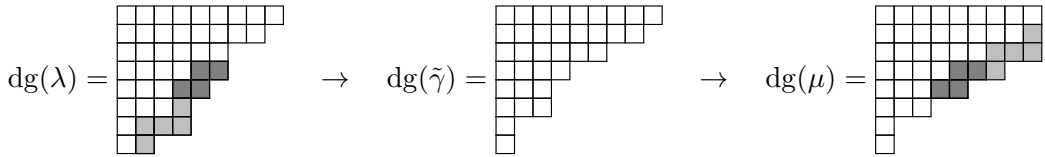
\[ \dg(\lambda)  = \y{9,8,6,6,5,4,4,2}*[*(lightgray)]{0,0,0,0,0,3+1,1+3,1+1}*[*(gray)]{0,0,0,4+2,3+2} 
\quad\to\quad \dg(\tilde{\ga})= \y{9,8,6,4,3,3,1,1} 
\quad\to\quad \dg(\mu)= 
\y{9,8,6,6,5,3,1,1}*[*(lightgray)]{0,8+1,6+3,6+1}*[*(gray)]{0,0,0,4+2,3+2} \] 
On one hand, $\dg(\ga)=\dg(\la)\cap\dg(\mu)$, and $R$ and $S$ are disjoint
rim-hooks that can be added outside the rim of $\dg(\gamma)$ to reach
$\dg(\la)$ or $\dg(\mu)$, respectively. On the other hand,
we get the rim-hook $\tilde{R}$ by adding to $R$ the cells on the inside
border of $\dg(\ga)$ that lead from the northeastern-most cell of $R$
to the southwestern-most cell of $S$ (these cells are shaded dark gray above).
Similarly, $\tilde{S}$ consists of $S$ along with these same cells,
while $\dg(\tilde{\ga})$ is $\dg(\ga)$ with these cells removed.
By translating the cases in the abacus-based proof back to partition
diagrams, one may check that these descriptions of $\ga$, $R$, $S$, 
$\tilde{\ga}$, $\tilde{R}$, and $\tilde{S}$ are valid 
whenever $G(\la,\mu)=\{\ga,\tilde{\ga}\}$ is nonempty.
\end{example}

\begin{remark}
Stanton and White~\cite[Thm. 4]{stanton-white} gave a bijective
proof that our matrix $A_n$ (counting signed RHT) satisfies the
sorting condition of Remark~\ref{rem:sort-mat}. This can also
be deduced algebraically from the observation that when $\sort(\al)
=\sort(\be)$, the power-sum symmetric functions $p_{\al}$ and $p_{\be}$
are equal and have the same Schur expansion. By Remark~\ref{rem:sort-mat},
the restriction of $A_n$ to a square $P(n)\times P(n)$ matrix has
inverse $B'_n$ with entries
\begin{equation}\label{eq:sq-inv2}
 B'_n(\la,\mu)
    =\sum_{\substack{\be\in C(n):\\ \sort(\be)=\la}} B_n(\be,\mu)
    =\sum_{\substack{\be\in C(n):\\ \sort(\be)=\la}} A_n(\mu,\be)/Z_{\be}
    =A_n(\mu,\la)\sum_{\substack{\be\in C(n):\\ \sort(\be)=\la}} Z_{\be}^{-1}.
\end{equation}
This classical result is usually stated as follows.
For a partition $\la$ containing $m_k$ copies of $k$ for each $k$,
define the integer $z_{\la}=\prod_{k\geq 1} m_k!k^{m_k}$.
Then the square matrix $A$ with $\la,\mu$-entry $\chi^{\la}_{\mu}$
has inverse with $\la,\mu$-entry $\chi^{\mu}_{\la}/z_{\la}$. 
$A$ is known to be the character table for the symmetric group $S_n$,
and this result is a special case of the orthogonality of irreducible
characters of a finite group. Part~(c) of Lemma~\ref{lem:harmonic} (below)
shows why~\eqref{eq:sq-inv2} agrees with the classical formulation.
\end{remark}

To state Lemma~\ref{lem:harmonic}, we need the following definitions.
Any $\si\in S_n$ can be written as a product of disjoint cycles.
Define the \emph{cycle partition} $\cycP(\si)\in P(n)$ 
to be the list of the lengths of these cycles (including $1$-cycles)
written in weakly decreasing order. Most permutations have several
different cycle notations, obtained by reordering cycles or starting
each cycle at a different position. Define the \emph{canonical cycle 
notation} for $\si$ by requiring that each cycle start at its minimum 
element and writing cycles with their minimum elements in decreasing order.

Define the \emph{cycle composition} $\cycC(\si)\in C(n)$ to be
the list of lengths of the cycles in this canonical cycle notation
for $\si$. For example, $\si=(5,2,1)(6,4,7)(9,3)(8)\in S_9$ has
canonical cycle notation $(8)(4,7,6)(3,9)(1,5,2)$, $\cycP(\si)=(3,3,2,1)$,
and $\cycC(\si)=(1,3,2,3)$.

\begin{lemma}\label{lem:harmonic}
(a)~For all $\la\in P(n)$, $n!/z_{\la}$ is the number of $\si\in S_n$
 with $\cycP(\si)=\la$.
\\ (b)~For all $\be\in C(n)$, $n!/Z_{\be}$ is the number of $\si\in S_n$
 with $\cycC(\si)=\be$.
\\ (c) For all $\la\in P(n)$, 
$\displaystyle{\sum_{\substack{\be\in C(n):\\ \sort(\be)=\la}} 
Z_{\be}^{-1}=z_{\la}^{-1}}$.
In other words, $z_{\la}$ is the harmonic mean of the $Z_{\be}$ as $\be$
ranges over all rearrangements of $\la$.
\end{lemma}
\begin{proof}
Part~(a) is a standard result giving the size of the conjugacy class
of $S_n$ indexed by $\la$; see, for example,~\cite[Thm. 7.115]{loehr-comb}.
To prove part~(b), we fix $\be\in C(n)$ and give the following
construction that builds all $\si\in S_n$ with $\cycC(\si)=\be$.
We use $\be=(1,3,2,3)$ as a running example to illustrate the construction.
For convenience, write $k=\ell(\be)$ and $B_i=\be_1+\cdots+\be_i$
for $i=1,2,\ldots,k$.

Begin with a list of cycles with $\be_1$ blanks in the first cycle,
$\be_2$ blanks in the second cycle, and so on. Our example begins
with $(\blk)(\blk,\blk,\blk)(\blk,\blk)(\blk,\blk,\blk)$.
Fill the blanks in the rightmost cycle first. The first blank must be $1$;
there are $n-1$ choices for the next value, $n-2$ choices for the value
after that, and so on; there are $n-(\be_k-1)=B_{k-1}+1$
choices for the last value in this cycle. In our example, we might pick $1$ 
(forced) then $5$ then $2$ to get $(\blk)(\blk,\blk,\blk)(\blk,\blk)(1,5,2)$.
Move left to the next empty cycle. The first blank must be the minimum
value not used already. Then there are $n-\be_k-1=B_{k-1}-1$
choices for the next value, $B_{k-1}-2$ choices for the
next value, and so on; there are 
$n-\be_k-(\be_{k-1}-1)=B_{k-2}+1$ choices for the last value.
In our example, we might pick $3$ (forced) then $9$ to get
$(\blk)(\blk,\blk,\blk)(3,9)(1,5,2)$. Continue filling the cycles from
right to left, noting that the first element in each new cycle must be
the minimum element in $[n]$ not already chosen. 

The number of ways to execute this construction is
the product of all integers in the list $n,n-1,n-2,\ldots,3,2,1$ excluding 
$n=B_k,B_{k-1},B_{k-2},\ldots,B_1=\be_1$.
The excluded integers correspond to the forced choices of the minimum
element in each new cycle. We see that the number of $\si\in S_n$
with $\cycC(\si)=\be$ is $\frac{n!}{B_1B_2\cdots B_k}=\frac{n!}{Z_{\be}}$,
as needed.

To prove part~(c), fix $\la\in P(n)$. Since the canonical cycle notation
of a permutation is unique, the set $\{\si\in S_n: \cycP(\si)=\la\}$
is the disjoint union of the sets $\{\si\in S_n: \cycC(\si)=\be\}$
as $\be$ ranges over all compositions in $C(n)$ that sort to $\la$.
By parts~(a) and~(b), we deduce $$n!/z_{\la}=
\sum_{\substack{\be\in C(n):\\ \sort(\be)=\la}} n!/Z_{\be}.$$
Dividing by $n!$ gives the result.
\end{proof}

\begin{example}
For $\la=(3,2,2)\in P(7)$, we have $z_{\la}=24$,
$Z_{(3,2,2)}=105$, $Z_{(2,3,2)}=70$, $Z_{(2,2,3)}=56$, and
$\frac{1}{105}+\frac{1}{70}+\frac{1}{56}=\frac{1}{24}$.
With the $7!$ factor included, this says $48+72+90=210$.
\end{example}

\section{Application 3: Incidence Matrices for Composition Refinement}
\label{sec:comp-ref}


For our third application, we use our local technique to prove a 
well-known inversion result for matrices whose entries record 
when one composition refines another. These matrices, along with
their weighted generalizations considered in Section~\ref{subsec:wtd-appl3},
are transition matrices between certain bases of quasisymmetric functions
and non-commutative symmetric functions (see Remark~\ref{rem:appl3}).
In this application, the local identities~\eqref{eq:gen-local} can be 
informally summarized by the slogan: 
``A local inverse of a composition of $L$ is a signed brick of length $L$.''

\subsection{Refinement Ordering on Compositions.}
\label{subsec:refine-order}

Let $\al,\be$ be compositions of $n$.  
Define $\al\leq\be$ to mean: there exist $i_k$ such that
$0 = i_0 < i_1 < i_2 < \ldots < i_{\ell(\be)} = \ell(\al)$ 
and $\be_k = \al_{i_{k-1}+1} + \ldots + \al_{i_{k}}$
for $k=1,2,\ldots,\ell(\be)$. When this holds, we say
\textit{$\al$ refines $\be$} and \textit{$\be$ coarsens $\al$}.
For example, $(2,1,1,2,1,3,1,1)\leq (3,4,3,2)$.
In general, we go from $\be$ to a composition refining $\be$
by replacing each part $\be_k$ by a composition of $\be_k$.
We go from $\al$ to a composition coarser than $\al$
by picking zero or more blocks of consecutive parts in $\al$
and replacing each block of parts by their sum.
The set $C(n)$ is partially ordered by the refinement relation.


We use tableau-like structures called \emph{compositional brick tabloids
(CBTs)} to visualize the refinement relation between compositions. 
These objects resemble the brick tabloids used in~\cite{eg-rem} to study
transition matrices between homogeneous symmetric functions and
elementary symmetric functions, but CBTs have simpler combinatorial structure.
See also~\cite{hicks} where similar objects are studied in connection with
the combinatorics of non-commutative symmetric functions.
Given $\al,\be\in C(n)$, a \emph{CBT of shape $\al$ and content $\be$}
is a filling of the diagram of $\al$ with $\be_k$ copies of 
$k$, for $1\leq k \leq \ell(\be)$, such that values weakly increase
from left to right in each row, and each value in any row is strictly
larger than the values appearing in all higher rows.
Let $\cbt(\al,\be)$ be the set of such objects.
We may interpret an element in $\cbt(\al, \be)$ as a 
tiling of $\dg(\al)$ by labeled bricks, where the $i$th brick
has length $\be_i$, satisfying the stated conditions on brick labels. 

\begin{example}\label{ex:CBT}
\boks{0.4}
    For $\al=(4,5,5,3)$ and $\be = (3,1,3,2,5,1,2)$,
the set $\cbt(\al, \be)$ consists of a single object shown here.
    \[ \yt{1112,33344,55555,677} \]
\end{example}

In general, given $\al,\be\in C(n)$, the only possible tiling in 
$\cbt(\al,\be)$ satisfying the 
ordering conditions is formed by laying down the bricks in order from $1$ to
$\ell(\be)$, working through $\dg(\al)$ from the top row down, and working 
left to right within each row. This tiling process succeeds if and only
if $\be$ refines $\al$. Thus, $|\cbt(\al,\be)|=\chi(\be\leq\al)$.
In the case where $\cbt(\al,\be)$ is nonempty, define the \emph{sign}
of the unique CBT in this set to be $(-1)^{\ell(\be)-\ell(\al)}$.
We may view the sign combinatorially by attaching a sign of $+1$ to
the first brick in each row and $-1$ to all other bricks, and letting
the sign of a tiling be the product of the signs of all its bricks.
The CBT in Example~\ref{ex:CBT} has sign $(-1)^3=-1$.

In this third application of the general setup of~\S\ref{subsec:setup},
we let $A_n$ be the $C(n)\times C(n)$ \emph{refinement poset incidence matrix}
with entries 
\begin{equation}\label{eq:refmatA}
 A_n(\la,\be)=\chi(\la\leq\be)=\sum_{T\in\cbt(\be,\la)} 1
\qquad\mbox{for $\la,\be\in C(n)$.}  
\end{equation}
We will prove that $A_n$ has inverse $B_n$ given by
\begin{equation}\label{eq:refmatB}
 B_n(\be,\mu)=(-1)^{\ell(\be)-\ell(\mu)}\chi(\be\leq\mu)
     =\sum_{T\in\cbt(\mu,\be)} \sgn(T)
\qquad\mbox{for $\be,\mu\in C(n)$.}  
\end{equation}

This is a well-known result giving the M\"obius function for the refinement
poset on $C(n)$ (or equivalently, the poset obtained by ordering
the set of subsets of $[n-1]$ by set inclusion). Our point here is to
showcase how this inversion result follows from a simple combinatorial
argument based on our local inversion technique. This same technique
proves a more subtle variation where matrix entries count weighted compositions
(see \S\ref{subsec:wtd-appl3}).

\begin{example}
    For $n = 4$, the incidence matrices are shown here:
    \begin{footnotesize}
    \[
    A_4 = \bbmatrix{~ & 4 & 31 & 22 & 211 & 13 & 121 & 112 & 1111\cr
    4 & 1 & 0 & 0 & 0 & 0 & 0 & 0 & 0 \cr
31 & 1 & 1 & 0 & 0 & 0 & 0 & 0 & 0 \cr
22 & 1 & 0 & 1 & 0 & 0 & 0 & 0 & 0 \cr
211 & 1 & 1 & 1 & 1 & 0 & 0 & 0 & 0 \cr
13 & 1 & 0 & 0 & 0 & 1 & 0 & 0 & 0 \cr
121 & 1 & 1 & 0 & 0 & 1 & 1 & 0 & 0 \cr
112 & 1 & 0 & 1 & 0 & 1 & 0 & 1 & 0 \cr
1111 & 1 & 1 & 1 & 1 & 1 & 1 & 1 & 1
    }
\qquad  B_4 = \bbmatrix{~ & 4 & 31 & 22 & 211 & 13 & 121 & 112 & 1111\cr
    4 & 1 & 0 & 0 & 0 & 0 & 0 & 0 & 0 \cr
31 & -1 & 1 & 0 & 0 & 0 & 0 & 0 & 0 \cr
22 & -1 & 0 & 1 & 0 & 0 & 0 & 0 & 0 \cr
211 & 1 & -1 & -1 & 1 & 0 & 0 & 0 & 0 \cr
13 & -1 & 0 & 0 & 0 & 1 & 0 & 0 & 0 \cr
121 & 1 & -1 & 0 & 0 & -1 & 1 & 0 & 0 \cr
112 & 1 & 0 & -1 & 0 & -1 & 0 & 1 & 0 \cr
1111 & -1 & 1 & 1 & -1 & 1 & -1 & -1 & 1
    } \]

    \end{footnotesize}
\end{example}

\begin{remark}\label{rem:appl3}
Several classical transition matrices for bases of QSym 
(the space of quasisymmetric functions), and NSym 
(the space of non-commutative symmetric functions)
can be inverted using this result. We briefly state the results here,
referring to~\cite{gelfand,hicks,qsym-book} for more details and definitions.
First, the monomial quasisymmetric basis $(M_{\al})$ and
Gessel's fundamental quasisymmetric basis $(F_{\al})$ of QSym satisfy
\[ F_{\la}=\sum_{\be\in C(n)} \chi(\be\leq\la)M_{\be} 
          =\sum_{\be\in C(n)} A_n(\be,\la)M_{\be}
\qquad\mbox{for all $\la\in C(n)$.} \]
The (transposed) inversion result for $A_n$ and $B_n$ therefore gives
\[ M_{\be}=\sum_{\mu\in C(n)} B_n(\mu,\be)F_{\mu}
        =\sum_{\mu\in C(n)} (-1)^{\ell(\mu)-\ell(\be)}\chi(\mu\leq\be)F_{\mu}
\qquad\mbox{for all $\be\in C(n)$.} \]
Second, the ribbon Schur basis $(\rr_{\al})$ and
the non-commutative complete basis $(\hh_{\al})$ of NSym satisfy 

\[ \hh_{\be}=\sum_{\al\in C(n)} \chi(\be\leq\al)\rr_{\al} \mbox{ and hence }
   \rr_{\be}=\sum_{\al\in C(n)} (-1)^{\ell(\be)-\ell(\al)}\chi(\be\leq\al)
 \hh_{\al} \quad\mbox{ for all $\be\in C(n)$.} \]
Third, by redistributing signs between the matrices $A_n$ and $B_n$,
it is routine to check that the $C(n)\times C(n)$ matrix given by
$\overline{A_n}(\la,\be)=(-1)^{n-\ell(\la)}\chi(\la\leq\be)$ is its own inverse.
This matrix is the transition matrix (in both directions) between
the NSym bases $(\hh_{\al})$ and $(\ee_{\al})$: 

\begin{equation}\label{eq:NSym-he}
 \hh_{\be}=\sum_{\mu\in C(n)} \overline{A_n}(\la,\be)\ee_{\la}
\quad\mbox{ and }\quad
   \ee_{\be}=\sum_{\mu\in C(n)} \overline{A_n}(\la,\be)\hh_{\la}
\quad\mbox{ for all $\be\in C(n)$}. 
\end{equation}
Fourth, the weighted variation in Section~\ref{subsec:wtd-appl3}
(with a sign adjustment) gives the transition matrices between
the NSym bases $(\textbf{h}_{\al})$ and $(\pow_{\al})$, namely:
\begin{equation}\label{eq:Nsym-h-pow}
\textbf{h}_\beta = \sum\limits_{\lambda\in C(n)} 
\dfrac{\chi(\lambda\leq \beta)}{Z_{\beta,\lambda}} \pow_\lambda 
\ \mbox{ and }\ 
\pow_\mu = \sum\limits_{\beta\in C(n)} (-1)^{\ell(\mu) - \ell(\beta)} 
\chi(\beta \leq \mu) L_{\mu,\beta} \,\textbf{h}_\beta
\quad\mbox{ for all $\be,\mu\in C(n)$,}
\end{equation}
where $Z_{\be,\la}$ and $L_{\mu,\be}$ are defined in \S\ref{subsec:wtd-appl3}.
\end{remark}

\subsection{Recursions for $A_n$ and $B_n$}
\label{subsec:refine-recs}
Given $\la,\be\in C(n)$ with $\be=(\be^*,L(\be))$, 
a CBT of shape $\be$ and content $\la$ (if it exists) consists of 
a CBT of shape $\be^*$ and content $\ga$ (where $\ga$ is
some prefix of $\la$), followed by the bottom row of length $L(\be)$
filled with bricks consisting of the suffix of $\la$ not in $\ga$.
For any $\la\in C(n)$ and $L\in\Z_{>0}$, define $S(\la,L)$ as follows.
If $\la$ has a suffix $\la_{k+1},\ldots,\la_{\ell(\la)}$ with sum $L$,
let $S(\la,L)=\{\ga\}$ where $\ga=(\la_1,\ldots,\la_k)$; 
otherwise let $S(\la,L)=\emptyset$.
The preceding discussion of CBT proves that
\begin{equation}\label{eq:A-rec3}
 A(\la,\be)=\sum_{\ga\in S(\la,L(\be))} A(\ga,\be^*),
\end{equation}
which is an instance of~\eqref{eq:A-rec} with $\wt_A(\la,\ga)=1$.
Note that if $\ga$ exists but does not refine $\be^*$, then $\la$
does not refine $\be$ and both sides of~\eqref{eq:A-rec3} are zero.
In Example~\ref{ex:CBT}, renaming the partitions as 
$\la=(3,1,3,2,5,1,2)$ and $\be=(4,5,5,3)$,
we have $L(\be)=3$, $\be^*=(4,5,5)$, $\ga=(3,1,3,2,5)$,
and $A(\la,\be)=1=A(\ga,\be^*)$.

Next, given $\be,\mu\in C(n)$ with $\be=(\be^*,L(\be))$, we may pass
from a signed CBT of shape $\mu$ and content $\be$ (if it exists) to a smaller
signed CBT of some shape $\de$ and content $\be^*$ by removing the final brick
(with length $L(\be)$ and label $\ell(\be)$) along with the cells in
$\dg(\mu)$ occupied by that brick. We obtain $\de$ from $\mu$ by subtracting
$L(\be)$ from the last part of $\mu$ and deleting the new part if it is zero.
In the case where the last row of $\dg(\mu)$ has more than one 
brick, the sign changes since we removed a brick of sign $-1$ from the last row.
In the case where the last row of $\dg(\mu)$ has a single brick, 
the sign does not change since we removed the first (and only) brick from the 
last row.  For $\mu=(\mu_1,\ldots,\mu_k)\in C(n)$ and $L>0$, use three cases to
define the set $T(\mu,L)$. If $L>\mu_k$, let $T(\mu,L)=\emptyset$.
If $L=\mu_k$, let $T(\mu,L)=\{\de\}$ with $\de=(\mu_1,\ldots,\mu_{k-1})$
and $\sgn(\mu,\de)=+1$. If $L<\mu_k$, let $T(\mu,L)=\{\de\}$ with
$\de=(\mu_1,\ldots,\mu_{k-1},\mu_k-L)$ with $\sgn(\mu,\de)=-1$.
The preceding discussion of CBT proves the recursion
\begin{equation}\label{eq:B-rec3}
 B(\be,\mu)=\sum_{\de\in T(\mu,L(\be))} \sgn(\mu,\de)B(\be^*,\de),
\end{equation}
which is an instance of~\eqref{eq:B-rec} with $\wt_B(\mu,\de)=\sgn(\mu,\de)$.
In Example~\ref{ex:CBT}, renaming the partitions as 
$\be=(3,1,3,2,5,1,2)$ and $\mu=(4,5,5,3)$,
we have $L(\be)=2<\mu_4$, $\be^*=(3,1,3,2,5,1)$, $\de=(4,5,5,1)$,
$\sgn(\mu,\de)=-1$, and $B(\be,\mu)=-1=-B(\be^*,\de)$.

\subsection{Proof of the Local Identity}
\label{subsec:prove-local3}

In this application, the local identity (\ref{eq:gen-local}) 
takes the following form.

\begin{theorem}\label{thm:local3}
Given $n>0$ and compositions $\la, \mu \in C(n)$, 
let $G(\la, \mu)$ be the set of compositions $\ga$
such that $\ga$ is a prefix of $\la$ and $\ga$
can be obtained from $\mu$ by decreasing the last part of $\mu$
by some positive amount $L$. Then
\begin{equation}\label{eq:local3}
 \sum\limits_{\ga \in G(\la, \mu)} \sgn(\mu,\ga) 
= \chi(\la = \mu).
\end{equation}
Therefore (by Theorem \ref{thm:main}), the matrices $A_n$ and $B_n$
in~\eqref{eq:refmatA} and~\eqref{eq:refmatB} are inverses of each other.
\end{theorem}
\begin{proof}
The left side of~\eqref{eq:local3} can be interpreted as the signed sum
of all possible ways of transforming $\mu=(\mu_1,\ldots,\mu_k)$ 
into $\la$ by the following two-step procedure. First, decrease the
last part of $\mu$ by some amount $L>0$, using the sign $+1$ if the
entire last part is removed and $-1$ otherwise, to reach some $\ga$. 
Second, append some composition of $L$ to $\ga$ to reach $\la$.
The intermediate object $\ga$ must be either $\mu^*=(\mu_1,\ldots,\mu_{k-1})$
or $(\mu_1,\ldots,\mu_{k-1},c)$ where $0<c<\mu_k$. 
The first object has sign $+1$ and may be used if and only if
$\mu^*$ is a prefix of $\la$. The second object has sign $-1$
and may be used if and only if $\mu^*$ is a prefix of $\la$ and $\la_k=c$. 

In the case where $\la=\mu$, we cannot use the second object since
$c<\mu_k=\la_k$. But the first object works since $\mu^*=\la^*$ is 
a prefix of $\la$. In this case, the left side of~\eqref{eq:local3} is 1,
as needed.  In the case where $\la\neq\mu$ and $\mu^*$ is 
not a prefix of $\la$, the sum on the left side of~\eqref{eq:local3} is 
vacuous and both sides are $0$. Finally, consider the case where
$\la\neq\mu$ and $\mu^*$ is a prefix of $\la$. Here $c=\la_k$ must satisfy
$0<c<\mu_k$, since $\la_k=\mu_k$ would contradict $\la\neq\mu$, while
$c=0$ or $\la_k>\mu_k$ would contradict $\la,\mu\in C(n)$. In this case,
$G(\la,\mu)=\{\mu^*,(\mu^*,\la_k)\}$. So both sides of~\eqref{eq:local3}
are $0$, as needed.

\end{proof}

To illustrate the final case of the proof,
take $\la=(4,1,3,2,1,3)$ and $\mu=(4,1,3,6)$.
Then $G(\la,\mu)=\{(4,1,3),(4,1,3,2)\}$.
On the other hand, $G(\mu,\mu)=\{(4,1,3)\}$.

\subsection{Weighted Variant of the Inversion Result.}
\label{subsec:wtd-appl3}

This subsection proves a weighted variation of the inversion
result in Theorem~\ref{thm:local3}. After rearranging
some signs, we obtain the transition matrices between the
bases $(\mathbf{h}_{\al})$ and $(\pow_{\al})$ of NSym,
as stated in~\eqref{eq:Nsym-h-pow} of Remark~\ref{rem:appl3}.

Suppose $\al,\be\in C(n)$ and $\be\leq\al$, so
$\cbt(\al,\be)$ contains a unique object $T$.
For $1\leq i\leq\ell(\al)$, let $\be^{(i)}\in C(\al_i)$
be the list of consecutive parts in $\be$ whose associated
bricks appear in $T$ in row $i$ of $\dg(\al)$.
For any composition $\ga=(\ga_1,\ga_2,\ldots,\ga_s)$,
define $Z_{\ga}=\ga_1(\ga_1+\ga_2)\cdots(\ga_1+\ga_2+\cdots+\ga_s)$
as in~\S\ref{subsec:scale_Z_beta}. 
When $\be\leq\al$, define $Z_{\al,\be}=\prod_{i=1}^{\ell(\al)} Z_{\be^{(i)}}$
and $L_{\al,\be}=\prod_{i=1}^{\ell(\al)} L(\be^{(i)})$.
Thus $Z_{\al,\be}$ is the product of the $Z$-factors arising from
the compositions in each row of $T$, while 
$L_{\al,\be}$ is the product of the lengths of the last bricks
in each row of $T$. In Example~\ref{ex:CBT}, where
$\al=(4,5,5,3)$ and $\be = (3,1,3,2,5,1,2)$, we compute
\[ Z_{\al,\be}=Z_{(3,1)}Z_{(3,2)}Z_{(5)}Z_{(1,2)}
               =(3\cdot 4)\cdot (3\cdot 5)\cdot (5)\cdot (1\cdot 3) =2700 \]
and $L_{\al,\be}=1\cdot 2\cdot 5\cdot 2=20$.

\begin{theorem}\label{thm:local3a}
For $\la,\be,\mu\in C(n)$, define
\[ A_n(\la,\be)=\chi(\la\leq\be)L_{\be,\la}\quad\mbox{ and }\quad
   B_n(\be,\mu)=(-1)^{\ell(\be)-\ell(\mu)}\chi(\be\leq\mu)Z_{\mu,\be}^{-1}. \]
The matrices $A_n$ and $B_n$ are inverses of each other.
\end{theorem}
\begin{proof}
We adapt the proof of Theorem~\ref{thm:local3} by incorporating the
extra weight factors in the entries of $A_n$ and $B_n$.
First consider the recursion~\eqref{eq:A-rec3}.
Say we remove the last row of a CBT $T$ of shape $\be$ and content $\la$
to get a CBT $T^*$ of shape $\be^*$ and content $\ga$. The last bricks
of $T$ and $T^*$ are the same except $T$ has a last brick of size
$L(\la)$ in its final row of length $L(\be)$.
Therefore, $L_{\be,\la}=L_{\be^*,\ga}L(\la)$. 
So recursion~\eqref{eq:A-rec3} becomes
\begin{equation}\label{eq:A-rec3a}
 A(\la,\be)=\sum_{\ga\in S(\la,L(\be))} L(\la)A(\ga,\be^*),
\end{equation}
which is an instance of~\eqref{eq:A-rec} with $\wt_A(\la,\ga)=L(\la)$.

Next consider the recursion~\eqref{eq:B-rec3}.
Say we remove the last brick in the last row of a CBT $T$ of shape
$\mu$ and content $\be$ to get a CBT $T^*$ of shape $\de$ and content $\be^*$.
All the factors in $Z_{\mu,\be}$ and $Z_{\de,\be^*}$ are the same except
the final row of $T$ contributes an extra factor $L(\mu)$ compared to $T^*$.
In other words, $Z_{\mu,\be}=Z_{\de,\be^*}L(\mu)$. 
The sign analysis is the same as before, so we get
\begin{equation}\label{eq:B-rec3a}
 B(\be,\mu)=\sum_{\de\in T(\mu,L(\be))} \sgn(\mu,\de)L(\mu)^{-1}B(\be^*,\de),
\end{equation}
which is an instance of~\eqref{eq:B-rec} with 
$\wt_B(\mu,\de)=\sgn(\mu,\de)L(\mu)^{-1}$.

Finally, taking the new weights into account, the local 
identity~\eqref{eq:local3} becomes
\begin{equation}\label{eq:local3a}
 \sum\limits_{\ga \in G(\la, \mu)} \frac{L(\la)}{L(\mu)}\sgn(\mu,\ga) 
= \chi(\la = \mu).
\end{equation}
This reduces to~\eqref{eq:local3} when $\la=\mu$.
When $\la\neq\mu$, the factor $L(\la)/L(\mu)$ can be brought outside the sum,
and the result again follows from~\eqref{eq:local3}.
Theorem~\ref{thm:local3a} follows from~\eqref{eq:local3a}
and Theorem~\ref{thm:main}.
\end{proof}


\section{Application 4: Brick Tabloids}
\label{sec:brick-tabl}

For our fourth application, we will prove inversion results for 
rectangular matrices that count certain weighted brick tabloids.
These are relevant to the transition matrices between the monomial 
basis and power-sum basis of symmetric functions. In this case,
the local identities~\eqref{eq:gen-local} can be informally summarized 
by the slogan: 
``A local inverse of a brick of size $L$ is a partition of $L$
(with suitable signs and weights).''

\subsection{Background on Brick Tabloids.}
\label{subsec:def-brick-tabl}

In this application, we often view integer partitions as multisets of
parts. For each partition $\la$ and $i>0$, define the \emph{multiplicity}
$m_i(\la)$ to be the number of times $i$ appears as a part in $\la$.
For partitions $\la$ and $\mu$, let $\la\uplus\mu$, $\la\cap\mu$,
and $\la\smin\mu$ be the partitions obtained by applying the standard
set operations to $\la$ and $\mu$ viewed as multisets. More precisely,
for each $i$, we have $m_i(\la\uplus\mu)=m_i(\la)+m_i(\mu)$,
$m_i(\la\cap\mu)=\min(m_i(\la),m_i(\mu))$, and
$m_i(\la\smin\mu)=\max(m_i(\la)-m_i(\mu),0)$.
Let $\la\subseteq\mu$ mean $m_i(\la)\leq m_i(\mu)$ for all $i>0$.
Let $i\in\mu$ mean $m_i(\mu)>0$; that is, $i$ appears as a part of $\mu$.

\begin{example}
For $\la=(4,2,2,1,1,1)$ and $\mu = (3,2,1,1)$, 
$\la\uplus\mu=(4,3,2,2,2,1,1,1,1,1)$, $\la\cap\mu=(2,1,1)$,
$\la\smin\mu=(4,2,1)$ and $\mu\smin\la = (3)$.
Here, and in general, we have $\la=(\la\smin\mu)\uplus (\la\cap\mu)$
and $\mu=(\mu\smin\la)\uplus(\la\cap\mu)$.
\end{example}

Following~\cite{eg-rem}, we define an \textit{ordered brick tabloid} (OBT)
of shape $\la\in P(n)$ and content $\be\in C(n)$ 
to be a filling of $\dg(\la)$ such that 
each label $i$ appears in $\be_i$ cells, all in the same row,
and labels in each row weakly increase from left to right.
Let $\obt(\la,\be)$ be the set of such objects, which may be 
viewed as tilings of the diagram of $\la$ using labeled horizontal bricks
of lengths $\be_1,\be_2,\ldots$.

\begin{example}\label{ex:obts}
\boks{0.4}
For $\be = (2,1,1,3)$ and $\la = (4,3)$, 
the three ordered brick tabloids in $\obt(\la,\be)$ are shown here.
    \[
    \yt{1123,444} \qquad \yt{2444,113} \qquad \yt{3444,112}
    \]
\end{example}

Recall the set $\cbt(\be,\al)$ of compositional brick tabloids
of shape $\be\in C(n)$ and content $\al\in C(n)$, which we defined 
in~\S\ref{subsec:refine-order}. For $\be\in C(n)$ and $\mu\in P(n)$,
define the set $\bt(\be,\mu)$ of \emph{brick tabloids of shape $\be$ and 
type $\mu$} to consist of all CBTs of shape $\be$ and content $\al$
where $\sort(\al)=\mu$.
For $T\in \cbt(\be, \al)$, define $\wt(T)=L_{\be,\al}$,
which is the product of the lengths of the last (rightmost) brick in each row.
For $\be\in C(n)$ and $\mu\in P(n)$, define 
$w_{\be, \mu} = \sum_{T\in\bt(\be,\mu)} \wt(T)$.
In the special case $\be=(n)$, define $W_{\mu}=w_{(n),\mu}=\sum_{\al} L(\al)$,
where we sum over all $\al\in C(n)$
with $\sort(\al)=\mu$.  Note that $W_{(n)}=n$.

\begin{example}\label{ex:bricktabloids}
\boks{0.4}
Let $\be= (3,1,2)$ and $\mu = (2,2,1,1)$. The set $\bt(\be,\mu)$ contains
the two objects shown here.
\[ \yt{122,3,44}\qquad \yt{112,3,44} \]
The objects have weights $2\cdot 1 \cdot 2 = 4$ and $1\cdot 1\cdot 2 = 2$, respectively, so $w_{\be,\mu}=6$.
\end{example}


\begin{lemma}\label{lem:local4main}
(a)~For all $\mu\in P(n)$ such that $\mu\neq(n)$,
$W_{\mu} = \sum\limits_{i\in \mu} W_{\mu\smin (i)}$.
\\ (b)~For all partitions $\mu$,
$W_{\mu}=\frac{|\mu|}{\ell(\mu)}\binom{\ell(\mu)}{m_1(\mu),m_2(\mu),\ldots}$.
\end{lemma}
\begin{proof}
To prove (a), fix $\mu\in P(n)$ with $\mu\neq(n)$. 
Let $S$ be the set of $\al\in C(n)$ with $\sort(\al)=\mu$.
For each $i$ appearing as a part of $\mu$, let $S_i=\{\al\in S:\al_1=i\}$,
and let $T_i$ be the set of $\ga\in C(n-i)$ with $\sort(\ga)=\mu\smin(i)$.
The map $\ga\mapsto (i,\ga)$ is a bijection from $T_i$ onto $S_i$ that
preserves the last part. Also $S$ is the disjoint union of the sets $S_i$.
Therefore,
\[ W_{\mu}=\sum_{\al\in S} L(\al)
  =\sum_{i\in\mu}\sum_{\al\in S_i} L(\al)
  =\sum_{i\in\mu}\sum_{\ga\in T_i} L(\ga)
  =\sum_{i\in\mu} W_{\mu\smin(i)}. \] 

We give a bijective proof of~(b) in the rearranged form
$\ell(\mu)W_{\mu}=|\mu|\binom{\ell(\mu)}{m_1(\mu),m_2(\mu),\ldots}$.
Write $n=|\mu|$.
The left side counts the set $S$ of fillings of a row $(n)$ with
a rearrangement of the parts of $\mu$ (the bricks) where one brick
has been marked (accounting for $\ell(\mu)$) and one cell in the
rightmost brick has been marked (accounting for the way brick tabloids
are weighted in $W_{\mu}$). The right side counts the set $T$
of fillings of a row $(n)$ with a rearrangement of the parts of $\mu$ 
where one of the $n$ cells has been marked. Define a bijection 
$g:T\rightarrow S$ as follows. For $t\in T$, suppose the marked cell
of $t$ lies in brick $b$. Interchange $b$ with the last brick in $t$,
moving the marked cell along with $b$, and also mark the brick 
(formerly the last brick) that got switched with $t$. This defines
$g(t)$. To get $g^{-1}(s)$ where $s\in S$, find the marked brick in 
$s$ and interchange it with the last brick in $s$ (which has one of
its cells marked). In the following example with $\mu=(3,3,2,2,1)$, 
we mark the cell in $t$ and $g(t)$ by $*$ and mark the switched brick 
in $g(t)$ by gray shading.
\[ g\Bigg(\yt{111{2^*}2344555}\Bigg) 
= \yt{111{*(lightgray) 2}{*(lightgray) 2}{*(lightgray) 2}344{5^*}5} \qedhere \]

\end{proof}

In this fourth application of the general setup of~\S\ref{subsec:setup},
we let $A_n$ be the $P(n)\times C(n)$ matrix with entries
\begin{equation}\label{eq:OBTmatA}
 A_n(\la,\be)=|\obt(\la,\be)|=\sum_{T\in\obt(\la,\be)} 1
\qquad\mbox{for $\la\in P(n)$, $\be\in C(n)$.}  
\end{equation}
For $\be =(\be_1, \ldots, \be_s)\in C(n)$, 
let $Z_{\be}=\be_1(\be_1+\be_2)(\be_1+\be_2+\be_3)
\cdots (\be_1+\be_2+\cdots+\be_s)$ as in~\S\ref{subsec:scale_Z_beta}.
We will prove that $A_n$ has inverse $B_n$ with entries given by
\begin{equation}\label{eq:OBTmatB}
 B_n(\be,\mu)=(-1)^{\ell(\mu)-\ell(\be)}\frac{w_{\be,\mu}}{Z_{\be}}
\qquad\mbox{for $\be\in C(n)$, $\mu\in P(n)$.}  
\end{equation}

\begin{remark}
It can be shown (cf.~\cite{eg-rem}) that $A_n(\la,\be)$ is
the coefficient of $m_{\la}$ in the monomial expansion 
of the power-sum symmetric function $p_{\be}$. Since 
$\sort(\al)=\sort(\be)$ implies $p_{\al}=p_{\be}$,
the matrix $A_n$ satisfies the sorting condition 
of Remark~\ref{rem:sort-mat}. Thus we get a formula for
the power-sum expansion of monomial symmetric functions
by converting $B_n$ in~\eqref{eq:OBTmatB} to the square 
matrix $B_n'$ defined in~\eqref{eq:sort-inv}.
If $\be,\ga\in C(n)$ both sort to a given $\nu\in P(n)$,
then $w_{\be,\mu}=w_{\ga,\mu}$, as can be seen by permuting
the rows of brick tabloids of shape $\be$ to get brick tabloids of shape $\ga$.
Using this and Lemma~\ref{lem:harmonic}(c), we obtain
\[ B_n'(\nu,\mu)=(-1)^{\ell(\mu)-\ell(\nu)}\frac{w_{\nu,\mu}}{z_{\nu}}
\qquad\mbox{for all $\nu,\mu\in P(n)$.}   \]
\end{remark}

\begin{example} For $n=4$, the matrices $A_n$ and $B_n$ are as follows:

\begin{footnotesize}
\[ A_4:\ \bbmatrix{~ & 4 & 31 & 22 & 211 & 13 & 121 & 112 & 1111\cr
4  &  1 & 1 & 1 & 1 & 1 & 1 & 1 & 1 \cr
31 & 0 & 1 & 0 & 2 & 1 & 2 & 2 & 4 \cr
22 &0 & 0 & 2 & 2 & 0 & 2 & 2 & 6 \cr
211&0 & 0 & 0 & 2 & 0 & 2 & 2 & 12 \cr
1111&0 & 0 & 0 & 0 & 0 & 0 & 0 & 24}
\quad B_4:\ \bbmatrix{~ & 4 & 31 & 22 & 211 & 1111\cr
4&1 & -1 & -{1}/{2} & 1 & -{1}/{4} \cr
31&0 & {1}/{4} & 0 & -{1}/{4} & {1}/{12} \cr
22&0 & 0 & {1}/{2} & -{1}/{2} & {1}/{8} \cr
211&0 & 0 & 0 & {1}/{12} & -{1}/{24} \cr
13&0 & {3}/{4} & 0 & -{3}/{4} & {1}/{4} \cr
121&0 & 0 & 0 & {1}/{6} & -{1}/{12} \cr
112&0 & 0 & 0 & {1}/{4} & -{1}/{8} \cr
1111& 0 & 0 & 0 & 0 & {1}/{24} } \]
\end{footnotesize}
The square versions are:
\[ A_4':\ \bbmatrix{~ & 4 & 31 & 22 & 211 & 1111\cr
4  &  1 & 1 & 1 & 1  & 1 \cr
31 & 0 & 1 & 0 & 2 & 4 \cr
22 &0 & 0 & 2 & 2  & 6 \cr
211&0 & 0 & 0 & 2  & 12 \cr
1111&0 & 0 & 0 & 0 & 24} \quad
 B_4':\ \bbmatrix{~ & 4 & 31 & 22 & 211 & 1111\cr
4  &  1 & -1 & -1/2 & 1 & -1/4 \cr
31 & 0 & 1 & 0 & -1 & 1/3 \cr
22 &0 & 0 & 1/2 & -1/2  & 1/8 \cr
211&0 & 0 & 0 & 1/2  & -1/4 \cr
1111&0 & 0 & 0 & 0 & 1/24} \]
\end{example}

\subsection{Recursion for $A_n$}
\label{subsec:OBT-recA}

Fix a partition $\la\in P(n)$ and a positive integer $L$.
Let $S(\la,L)$ be the set of partitions $\ga$ that can be obtained
by replacing one part $i\geq L$ in $\la$ by $i-L$ and sorting the parts.
We can identify $i$ given $\la$ and $\ga$ as the sole member of the
multiset $\la\smin\ga$. Define
$\wt_A(\la,\ga)=m_i(\la)=m_{\la\smin\ga}(\la)$ to be the number of times 
the part value $i$ that got replaced appears in $\la$. 
In this setting, the general recursion~\eqref{eq:A-rec} takes the following
form.

\begin{lemma}\label{lem:A-rec4}
For all $\la\in P(n)$ and $\be\in C(n)$,
\begin{equation}\label{eq:A-rec4}
 |\obt(\la,\be)|=\sum_{\ga\in S(\la,L(\be))} 
    m_{\la\smin\ga}(\la)\left|\obt(\ga,\be^*)\right|.
\end{equation}
\end{lemma}
\begin{proof}
Informally, we obtain the recursion by removing the largest-labeled 
brick in an OBT of shape $\la$ and content $\be$. 
However, there is some extra complexity since we must sort
the new diagram so it still has partition shape. Formally, we proceed
by defining a bijection $F:\obt(\la,\be)\rightarrow
 \bigcup_{\ga\in S(\la,L(\be))} [\wt_A(\la,\ga)]\times \obt(\ga,\be^*)$.
Write $\be=(\be_1,\ldots,\be_s)=(\be^*,L)$ with $L=L(\be)$.
Given an input $T\in\obt(\la,\be)$, the brick with largest label $s$ and 
size $L$ must be at the end of some row of $\dg(\la)$ of some length $i\geq L$.
Suppose the brick appears in the $k$th highest row of length $i$
in the diagram, where $1\leq k\leq m_i(\la)$.
Delete this brick and its associated cells, and move the truncated row
(along with the remaining bricks in it) to become the highest row
of length $i-L$ in the new diagram. The new diagram is an OBT $T^*$
of some shape $\ga\in S(\la,L)$ such that $\wt_A(\la,\ga)=m_i(\la)$.
Define $F(T)=(k,T^*)$.

To invert $F$, start with any $(k,T^*)$ in the codomain of $F$,
where $T^*$ has shape $\ga\in S(\la,L)$. Here, $i$ must
be the unique part in $\la\smin\ga$, and $\ga$ must have at least
one part of size $i-L$. Put a brick of label $s$ and size $L$ at the end
of the highest row of size $i-L$ in $T^*$. Move this enlarged row up
to become the $k$th highest row among the rows of length $i$ in a new OBT that
has shape $\la$ and content $\be$. It is routine to check that
these steps define the two-sided inverse of $F$.
\end{proof}

\begin{example}
For $\la = (3,3,2)$ and $\be = (2,1,3,2)$, the set $S(\la,L(\be))$
contains two partitions: $\ga=(3,2,1)$ with $\wt_A(\la,\ga)=m_3(\la)=2$,
and $\de=(3,3)$ with $\wt_A(\la,\de)=m_2(\la)=1$.
The four objects in $\obt(\la,\be)$ are shown here.
\[ T_1=\yt{244,333,11},\quad T_2=\yt{333,244,11},\quad
T_3=\yt{112,333,44},\quad T_4= \yt{333,112,44} \]
On the other hand, $\obt(\ga, \be^*)=\{T_5\}$ 
and $\obt(\de,\be^*)=\{T_6,T_7\}$, where
\[ T_5=\yt{333,11,2},\quad T_6=\yt{112,333},\quad T_7=\yt{333,112}. \]
In this example, the bijection $F$ acts as follows:
\[ F(T_1)=(1,T_5),\quad F(T_2)=(2,T_5),\quad 
   F(T_3)=(1,T_6),\quad F(T_4)=(1,T_7). \]
\end{example}

\subsection{Recursion for $B_n$}
\label{subsec:OBT-recB}

For $\mu\in P(n)$ and $L>0$, define 
$T(\mu,L)=\{\de\in P(n-L): \de\subseteq\mu\}$, which is the set 
of partitions of $n-L$ that can be obtained from $\mu$ by 
removing some parts that sum to $L$. In other words, 
for each $\de\in T(\mu,L)$, $\eps=\mu\smin\de$ is a partition of $L$
with $\eps\subseteq\mu$. For such $\de$, define
\begin{equation}\label{eq:wtB-4}
 \wt_B(\mu,\de)=\frac{(-1)^{\ell(\mu)-\ell(\de)-1}W_{\mu\smin\de}}{|\mu|}.
\end{equation}
In this setting, to verify the general recursion~\eqref{eq:B-rec},
we must prove: for $\be\in C(n)$ and $\mu\in P(n)$,
\begin{equation}\label{eq:B-rec4}
(-1)^{\ell(\mu)-\ell(\be)}\frac{w_{\be,\mu}}{Z_{\be}}
=\sum_{\de\in T(\mu,L(\be))} (-1)^{\ell(\mu)-\ell(\de)-1+\ell(\de)-\ell(\be^*)}
   \frac{W_{\mu\smin\de}w_{\be^*,\de}}{|\mu|Z_{\be^*}}. 
\end{equation}
The power of $-1$ is the same on both sides since $\ell(\be)=\ell(\be^*)+1$.
The denominators match as well (independent of $\de$) since
$|\mu|Z_{\be^*}=nZ_{\be^*}=|\be|Z_{\be^*}=Z_{\be}$. 
Since $\mu\smin\de$ is a partition of $L(\be)$, we are reduced to checking
\begin{equation}\label{eq:B-rec4b}
 w_{\be,\mu}= \sum_{\de\in T(\mu,L)} w_{(L),\mu\smin\de}w_{\be^*,\de}
\quad\mbox{ where $L=L(\be)$.} 
\end{equation}
This follows by mapping a brick tabloid $T$ counted by $w_{\be,\mu}$
to the pair of brick tabloids $(T',T^*)$, where $T'$ is the last row
of $T$ and $T^*$ consists of all higher rows of $T$. Note that 
$T^*$ is a brick tabloid of shape $\be^*$ and type $\de$ for some
$\de\in T(\mu,L)$, whereas $T'$ must then be a brick tabloid of
shape $(L)$ and type $\mu\smin\de$. Taking products of the lengths
of the rightmost bricks, we have $\wt(T)=\wt(T')\wt(T^*)$.
Thus the map $T\mapsto (T',T^*)$ is a weight-preserving bijection,
and~\eqref{eq:B-rec4b} follows.

\subsection{Proof of the Local Identity}
\label{subsec:prove-local4}

In this application, the local identity~\eqref{eq:gen-local}
takes the following form.

\begin{theorem}\label{thm:local4}
Given $n>0$ and partitions $\la,\mu\in P(n)$, let $G(\la,\mu)$ 
be the set of partitions $\ga$ that can be obtained by decreasing 
one part of $\la$ by some $L>0$ (then sorting) and by removing one or more
parts from $\mu$. Then
\begin{equation}\label{eq:local4}
\sum\limits_{\ga\in G(\la,\mu)} \dfrac{ m_{\la\smin\ga}(\la)
(-1)^{\ell(\mu)-\ell(\ga)-1} W_{\mu\smin\ga}}{n} = \chi(\la = \mu).
\end{equation}
Therefore (by Theorem~\ref{thm:main}), the rectangular matrix $A_n$ of
OBT counts has a right-inverse $B_n$ defined by~\eqref{eq:OBTmatB}.
\end{theorem}
\begin{proof}
First consider the case $\la=\mu$.
For each part $i\in\la$, let $\ga^i=\la\smin(i)$ be $\la$ with a single
copy of the part $i$ deleted. Evidently $\ga^i\in G(\la,\la)$ for all
$i\in\la$, and we claim these are the only elements in $G(\la,\la)$.
To see why, fix $\ga\in G(\la,\la)$. On one hand, we obtain $\ga$
from $\la$ by decreasing some $\la_j=i$ by some $L>0$. On the other hand,
$\ell(\ga)<\ell(\mu)=\ell(\la)$. If $L<i$, then we would have
$\ell(\ga)=\ell(\la)$, which is impossible. So $L=i$ and $\ga=\ga^i$. 
The left side of~\eqref{eq:local4} becomes
$\sum_{i\in\la} \frac{m_i(\la)(-1)^{1-1}W_{(i)}}{n}
 =\sum_{i\in\la} \frac{im_i(\la)}{n}=1$, since $|\la|=n$.

For the case $\la\neq \mu$, identity~\eqref{eq:local4} holds vacuously
when $G(\la,\mu)=\emptyset$, so assume $G(\la,\mu)$ is nonempty.
This means there is a way to go from $\la$ to $\mu$ by first 
subtracting some $L>0$ from a single part in $\la$ to get $\ga$,
then merging some partition $\nu$ of $L$ with $\ga$ to reach $\mu$.
In this situation, we claim $\la\smin\mu$ must equal the single-part
partition $(i)$ for some $i$. 
On one hand, $\la\smin\mu$ cannot be empty since $|\la|=|\mu|$ 
and $\la\neq\mu$. On the other hand, suppose the multiset $\la\smin\mu$ 
contained two parts $a$ and $b$ (where $a=b$ could occur). 
Decrementing the part $a$ by $L$ and merging with $\nu$ would produce a 
partition with too many copies of $b$ (compared to $\mu$).
Decrementing the part $b$ by $L$ and merging with $\nu$ would produce a 
partition with too many copies of $a$ (compared to $\mu$).
This proves the claim.

With the claim $\la\smin\mu=(i)$ established, 
we can describe all elements of $G(\la,\mu)$. 
Write $\mu\smin\la=\rho$, which must be a partition of $i$ different
from $(i)$ since $|\la|=|\mu|$ but $\la\neq\mu$.
To get an element of $G(\la,\mu)$,
we must decrease one of the copies of $i$ in $\la$ by some amount $L$
since that is the only way to get rid of the extra copy of $i$.
We could choose $L=i$, leading to $\ga=\la\smin(i)=\la\cap\mu\in G(\la,\mu)$,
then merge $\ga$ with $\rho$ to reach $\mu$.
Or we could choose any $L<i$ such that $i-L\in\rho$,
leading to $\ga=(\la\cap\mu)\uplus(i-L)\in G(\la,\mu)$,
then merge $\ga$ with the remaining parts in $\rho$ to reach $\mu$.
Putting all of this information into the left side of~\eqref{eq:local4}, 
the numerator becomes
\[ m_i(\la)(-1)^{\ell(\rho)-1}W_{\rho}
 +\sum_{L:\ i-L\in\rho} 
m_i(\la)(-1)^{\ell(\rho)-2}W_{\rho\smin(i-L)}
=m_i(\la)(-1)^{\ell(\rho)-1}\left[W_{\rho}-\sum_{j\in\rho} W_{\rho\smin(j)}
\right].\]
By Lemma~\ref{lem:local4main}(a), the term in brackets is zero,
so~\eqref{eq:local4} holds in this case.
\end{proof}

\begin{example}
For $\la=(5,2,2,1)$ and $\mu=(3,2,2,1,1,1)$, we have
$\la\smin\mu=(5)$, $\rho=\mu\smin\la=(3,1,1)$, and
$G(\la,\mu)=\{(2,2,1),(3,2,2,1),(2,2,1,1)\}$. 
The left side of~\eqref{eq:local4} is
\[ \frac{1(-1)^2W_{(3,1,1)}+1(-1)^1W_{(1,1)}+1(-1)^1W_{(3,1)}}{10}
  =\frac{5-1-4}{10}=0. \]
\end{example}

\section{Automatically Constructing Bijective Proofs of $AB=I$}
\label{sec:automatic-bij}

Returning to our general framework, it is often the case that the
entries in the matrices $A_n$ and $B_n$ count signed, weighted collections
of combinatorial objects. Suppose we have bijective proofs of 
recursion~\eqref{eq:A-rec} for the entries of $A_n$,
recursion~\eqref{eq:B-rec} for the entries of $B_n$, and 
the local identities~\eqref{eq:gen-local}. 
Then we can convert the algebraic proof of Theorem~\ref{thm:main}
to a bijective proof of the matrix identities $A_nB_n=I_{R(n)}$.
We sketch the general approach here and provide two concrete examples
(revisiting Applications~1 and~2) in the following subsections.

Suppose $A(\la,\be)$ is the sum of signed, weighted objects
in a set $\mcA_{\la,\be}$ for each $\la\in R(n)$ and $\be\in C(n)$.
Suppose $B(\be,\mu)$ is the sum of signed, weighted objects
in a set $\mcB_{\be,\mu}$ for each $\be\in C(n)$ and $\mu\in R(n)$.
By definition of matrix multiplication,
the $\la,\mu$-entry of $AB$ is the signed, weighted sum of all ordered 
pairs $(S,T)$ where $S\in\mcA_{\la,\be}$ for some $\be\in C(n)$,
  and $T\in\mcB_{\be,\mu}$ for the same choice of $\be$.
Let $\mcP_{\la,\mu}$ be the set of all such pairs.
We seek to construct sign-reversing, weight-preserving involutions
$\mcI_{\la,\mu}:\mcP_{\la,\mu}\rightarrow\mcP_{\la,\mu}$
where $\mcI_{\la,\mu}$ has no fixed points for all $\la\neq\mu$ in $R(n)$,
  and $\mcI_{\la,\la}$ has exactly one fixed point (with signed weight $+1$)
for all $\la\in R(n)$.

Tracing through the argument leading to~\eqref{eq:gen-intmed}
produces the following informal recursive description of $\mcI_{\la,\mu}$.
Given $(S,T)\in\mcP_{\la,\mu}$ with $S\in\mcA_{\la,\be}$
and $T\in\mcB_{\be,\mu}$ for some $\be=(\be^*,L(\be))$,
use the known bijective proofs of~\eqref{eq:A-rec} and~\eqref{eq:B-rec}
to remove some local structures from $S$ and $T$ to get
smaller objects $S^*\in\mcA_{\ga,\be^*}$ and $T^*\in\mcB_{\be^*,\de}$
for some $\ga,\de\in R(n-L(\be))$. In the case $\ga\neq\de$,
recursively apply $\mcI_{\ga,\de}$ to $(S^*,T^*)$ to find
the matching object $(S',T')$ of opposite sign. In the new object,
$\be^*$ may be replaced by some new composition $\be'\in C(n-L(\be))$.
Next, restore to $S'$ and $T'$ the same local structures that were removed 
from $S$ and $T$ to reach the final output $\mcI_{\la,\mu}(S,T)$.

In the case $\ga=\de$, we may sometimes still be able to recursively
apply $\mcI_{\ga,\de}$ to obtain a cancellation as in the first case.
However, it may be that $(S^*,T^*)$ is the fixed point for $\mcI_{\ga,\de}$.
There are two subcases here. If $\la\neq\mu$, then we use the given
bijective proof of the local identity to obtain a matching object of
opposite sign that will cancel with $(S,T)$. If $\la=\mu$, the local bijection
either yields another cancellation or tells us that $(S,T)$ is the
fixed point for $\mcI_{\la,\la}$.

Frequently, we can unroll the recursive description of $\mcI_{\la,\mu}$ to 
obtain an iterative algorithm to compute this map. 
Suppose we can explicitly describe the unique fixed point for
each $\mcI_{\la,\la}$; call this fixed point the \emph{survivor of shape $\la$}.
Given a non-survivor $(S,T)\in\mcP_{\la,\mu}$, 
repeatedly strip off the last local
structures from $S$ and $T$ until reaching an intermediate object
that is the survivor for some shape $\ga\in R(k)$. 
Apply the bijective proof of the local identity,
using the most recently removed local structure
to determine what $\la$ and $\mu$ to use in~\eqref{eq:gen-local}.
This bijection replaces the survivor of shape $\ga$ 
(and the next local structure just outside it)
by the survivor of some shape $\tilde{\ga}\in R(k)$ 
(and suitable new local structure just outside it).
To finish, act on this object by restoring the rest of the local structures 
that were removed from $S$ and $T$ to get the object $\mcI_{\la,\mu}(S,T)$
that cancels with $(S,T)$ in $\mcP_{\la,\mu}$.

Sometimes (as in Application~2 below), this approach can be 
generalized to the case where there is more than one surviving fixed point
in $\mcI_{\la,\la}$. This situation might occur if we need to rescale
one of the matrices $A_n$ or $B_n$ to ensure all entries are integers.

\subsection{Application~1: Bijective Inversion of the Kostka Matrix.}
\label{subsec:invert-Kostka}

We illustrate the general construction by developing a bijective proof
of $A_nB_n=I_{P(n)}$, where $A_n$ is the rectangular Kostka matrix from
Section~\ref{sec:kostka}. In this application, $\mcA_{\la,\be}$ is 
the (unsigned, unweighted) set $\ssyt(\la,\be)$ of semistandard Young
tableaux of shape $\la$ and content $\be$; and $\mcB_{\be,\mu}$ is
the (signed, unweighted) set $\srht(\mu,\be)$ of special rim-hook tableaux 
of shape $\mu$ and content $\be$. The set $\mcP_{\la,\mu}$ consists of
all pairs $(S,T)$ where $S\in\ssyt(\la,\be)$ and $T\in\srht(\mu,\be)$
for some composition $\be$. For each partition $\la$, let
$S_{\la}$ be the filling of $\dg(\la)$ where every cell in row $i$
contains the value $i$. The object $S_{\la}$ belongs to both
$\ssyt(\la,\la)$ and $\srht(\la,\la)$ and has positive sign.
It is routine to check that $\mcP_{\la,\la}$ consists of the sole
object $(S_{\la},S_{\la})$, which is the survivor for shape $\la$ 
in this application. For example, the survivor for $\la=(5,3,3,2)$ is
\[ \left( 
  \tableau{1&1&1&1&1 \\ 2&2&2 \\ 3&3&3 \\ 4&4}_,\quad
  \tableau{1&1&1&1&1 \\ 2&2&2 \\ 3&3&3 \\ 4&4} \right). \]

The involution $\mcI_{\la,\mu}$ acts on input $(S,T)\in\mcP_{\la,\mu}$ 
as follows.
If $\la=\mu$ and $(S,T)$ is the survivor for $\la$, then this is a fixed point.
Otherwise let $\be=(\be_1,\ldots,\be_\ell)$ be the content composition for $S$ 
and $T$.  Remove the outermost horizontal strip from $S$ 
(meaning the $\be_\ell$ cells in the filling of $\dg(\la)$ labeled $\ell$),
and remove the last special rim-hook from $T$ 
(meaning the $\be_\ell$ cells in the filling of $\dg(\mu)$ labeled $\ell$).
Do this repeatedly until reaching a survivor object of some shape $\ga$,
where $\ga$ might be empty. Suppose this survivor was reached after 
removing the cells labeled $k$ from $S$ (these cells forming
a horizontal strip $\eta$ of size $\be_k$) and
removing the cells labeled $k$ from $T$ (these cells forming 
a signed special rim-hook $\rho$ of size $\be_k$).
Let $\bla$ be the partition consisting of the cells in $\ga$ and $\eta$,
and let $\bmu$ be the partition consisting of the cells in $\ga$ and $\rho$.
Then the set $G(\bla,\bmu)$ (defined in Theorem~\ref{thm:local1})
is nonempty. The proof of that theorem shows that $G(\bla,\bmu)$ consists
of exactly two oppositely-signed objects and describes how to go from
one of these objects to the other. Let $\tilde{\ga}$ be the other object
appearing with $\ga$ in $G(\bla,\bmu)$. 
To make $(\tilde{S},\tilde{T}) =\mcI_{\la,\mu}(S,T)$,
start with the unique survivor of shape $\tilde{\ga}$.
Using the next label $\tilde{k}$ not appearing in $\tilde{\ga}$,
add a horizontal strip to $\tilde{S}$ to reach $\bla$ and
add a special rim-hook to $\tilde{T}$ to reach $\bmu$
(which is possible by definition of $G(\bla,\bmu)$).
Then restore the previously removed horizontal strips and
special rim-hooks of lengths $\be_{k+1},\ldots,\be_{\ell}$
to the fillings $\tilde{S}$ and $\tilde{T}$ (respectively)
in the same positions they occupied in $S$ and $T$.
Here the labels (originally $k+1,k+2,\ldots$) of the restored
structures get renumbered to be $\tilde{k}+1,\tilde{k}+2,\ldots$.

\begin{example}
Here are three examples of the action of $\mcI_{\la,\mu}$. First,
\begin{equation}\label{eq:ER-map1}
 \left(
\tableau{1&1&3\\
         2&2&4\\
         4&4}_,\quad
\tableau{1&1\\
         2&2\\
         3&4\\
         4&4}\,\,\right)
\quad \stackrel{\mcI_{\la,\mu}}{\longmapsto} \quad
\left(
\tableau{1&1&2\\
         2&2&3\\
         3&3}_,\quad
\tableau{1&1\\
         2&2\\
         2&3\\ 
         3&3}\,\,\right). 
\end{equation}
In this example, we remove the $4$s from the input objects and 
get a non-survivor.  Then we remove the $3$s from both objects
and reach the survivor for shape $\ga=(2,2)$. 
We compute $\bla=(3,2)$, $\bmu=(2,2,1)$, $\tilde{\ga}=(2)$
from the proof of Theorem~\ref{thm:local1}.
Starting with the survivor for $\tga$, we fill
$\dg(\bla)\setminus\dg(\tilde{\ga})$ with a horizontal strip of three $2$s, 
and we fill
$\dg(\bmu)\setminus\dg(\tilde{\ga})$ with a special rim-hook of three $2$s.
We then restore the last horizontal strip and special rim-hook,
renumbered to contain $3$s instead of $4$s.
The original content $\be=(2,2,1,3)$ changes 
to new content $\tilde{\be}=(2,3,3)$. 

For our second example, we compute
\begin{equation}\label{eq:ER-map2}
\left( \tableau{
 1&1&1&1&1&4&4\\
 2&2&2&4&4&6&6\\
 3&3&3&5\\
 4&4&6}_,\quad \tableau{
 1&1&1&1&1 \\
 2&2&2&4 \\
 3&3&3&4 \\
 4&4&4&4 \\
 5&6 \\
 6&6}\,\,\right)
\quad \stackrel{\mcI_{\la,\mu}}{\longmapsto} \quad
\left( \tableau{
 1&1&1&1&1&4&4\\
 2&2&2&2&4&6&6\\
 3&3&3&5\\
 4&4&6}_,\quad \tableau{
 1&1&1&1&1 \\
 2&2&2&2 \\
 3&3&3&4 \\
 4&4&4&4 \\
 5&6 \\
 6&6}\,\,\right).
\end{equation}
Here, $\be=(5,3,3,6,1,3)$, $\ga=(5,3,3)$, 
$\bla=(7,5,3,2)$, $\bmu=(5,4,4,4)$, $\tilde{\ga}=(5,4,3)$,
and $\tilde{\be}=(5,4,3,5,1,3)$.

For our third example, we compute
\begin{equation}\label{eq:ER-map3}
\left( \tableau{
 1&1&1&1&1&3&4\\
 2&2&2&3&3&6&6\\
 3&3&4&5\\
 4&4&6}_,\quad \tableau{
 1&1&1&1&1 \\
 2&2&2&3 \\
 3&3&3&3 \\
 4&4&4&4 \\
 5&6 \\
 6&6}\,\,\right)
\quad \stackrel{\mcI_{\la,\mu}}{\longmapsto} \quad
\left( \tableau{
 1&1&1&1&1&3&4\\
 2&2&2&2&3&6&6\\
 3&3&4&5\\
 4&4&6}_,\quad \tableau{
 1&1&1&1&1 \\
 2&2&2&2 \\
 3&3&3&3 \\
 4&4&4&4 \\
 5&6 \\
 6&6}\,\,\right).
\end{equation}
Here, $\be=(5,3,5,4,1,3)$, $\ga=(5,3)$, 
$\bla=(6,5,2)$, $\bmu=(5,4,4)$, $\tilde{\ga}=(5,4)$,
and $\tilde{\be}=(5,4,4,4,1,3)$.
\end{example}

\begin{remark}
The involutions $\mcI_{\la,\mu}$ constructed here are \emph{canonical}
in the technical sense that they do not rely on arbitrary choices.
On one hand, each set $\mcP_{\la,\la}$ contains a unique survivor
with no cancellation needed. On the other hand, for each $\la\neq\mu$
such that the set $G(\la,\mu)$ is nonempty, this set contains exactly
one positive object and exactly one negative object. Thus the involution
$\ga\leftrightarrow\tilde{\ga}$ on this set is uniquely determined.
Our general framework assembles the global maps $\mcI_{\la,\mu}$ from 
these canonical local ingredients without requiring any further choices.
\end{remark}

\begin{remark}
E\u{g}ecio\u{g}lu and Remmel~\cite{inv-kostka} gave a bijective proof
of $K_nK_n^{-1}=I_{P(n)}$ for the square $P(n)\times P(n)$ Kostka
matrices $K_n$. Their bijections differ from ours in several respects.
First, as a notational matter, they write partition diagrams in French
notation (longest row at the bottom) and use line segments rather than
numbered cells to display special rim-hooks. Second, their proof uses
a slightly different combinatorial model for the $\la,\mu$-entry of 
$K_nK_n^{-1}$ compared to our set $\mcP_{\la,\mu}$. Specifically,
they consider pairs $(S,T)$ where $S$ is a SSYT of shape $\la$
and $T$ is an SRHT of shape $\mu$ such that, for each $k\geq 1$,
the horizontal strip in $S$ filled with the label $k$ has size
equal to the length of the special rim-hook in $T$ starting in the $k$th row
of $\dg(\mu)$; this length is zero if there is no such rim-hook. 
In particular, the content of $S$ is a weak composition that might have 
some parts equal to zero. The validity of their model rests on the
fact that $|\ssyt(\la,\al)|=|\ssyt(\la,\be)|$ for any weak compositions
$\al,\be$ that are rearrangements of each other. Their involution requires
(as a subroutine) a bijection that interchanges the frequencies of 
$i$ and $i+1$ in a SSYT. In contrast, our bijections for
the rectangular Kostka matrices do not rely on these ingredients.

E\u{g}ecio\u{g}lu and Remmel~\cite[pp. 72--73]{inv-kostka} give the 
following example of their involution (denoted here by $\mcER_{\la,\mu}$;
we use our diagram conventions to facilitate comparison):
\begin{equation}
 \left(
\tableau{1&1&3\\
         2&2&4\\
         4&4}_,\quad
\tableau{1&1\\
         2&2\\
         3&4\\
         4&4}\right)
\quad \stackrel{\mcER_{\la,\mu}}{\longleftrightarrow} \quad
\left(
\tableau{1&1&3\\
         3&3&4\\
         4&4}_,\quad
\tableau{1&1\\
         3&3\\
         3&4\\ 
         4&4}\right). 
\end{equation}
We use the label $k$ for the special rim-hook starting in row $k$ from the top.
Their output has content $(2,0,3,3)$, which we do not allow. But comparing
to~\eqref{eq:ER-map1}, we observe that the actions of the maps $\mcI_{\la,\mu}$ 
and $\mcER_{\la,\mu}$ on this input essentially agree, 
after we adjust for different conventions regarding the content.
However, there are inputs on which the two involutions do not agree even in
this extended sense. For example, $\mcER_{\la,\mu}$ matches the two objects
on the left sides of~\eqref{eq:ER-map2} and~\eqref{eq:ER-map3} with
each other, as well as matching the two objects on the right sides.
\end{remark}

\subsection{Bijective Matrix Inversion in Application~2.}
\label{subsec:appl2-bij}

Getting a bijective proof of $A_nB_n=I_{P(n)}$ in Application~2
is more challenging, since the $B_n$ matrix defined in~\S\ref{sec:RHT}
has fractional entries. As a first step, we introduce the rescaled
matrix $\bB_n=n!B_n$ and prove $A_n\bB_n=n!I_{P(n)}$ instead.
In this case, $\mcA_n(\la,\be)$ is the signed set $\rht(\la,\be)$
of rim-hook tableaux of shape $\la$ and content $\be$.
Note that $\bB_n(\be,\mu)=\sum_{T\in\rht(\mu,\be)} \frac{n!}{Z_{\be}}\sgn(T)$.
Using Lemma~\ref{lem:harmonic}(b), we may take $\mcB_n(\be,\mu)$ to be 
the signed set consisting of pairs $(T,\si)$ such that $T\in\rht(\mu,\be)$ and 
$\si\in S_n$ has $\cycC(\si)=\be$. For $\la,\mu\in P(n)$,
$\mcP_{\la,\mu}$ is the set of triples $(S,T,\si)$ where, for some
composition $\be$, $S\in\rht(\la,\be)$, $T\in\rht(\mu,\be)$, $\si\in S_n$,
and $\cycC(\si)=\be$. The sign of $(S,T,\si)$ is $\sgn(S)\sgn(T)$.

The involutions $\mcI_{\la,\mu}$ must have no fixed points when $\la\neq\mu$,
but $\mcI_{\la,\la}$ should have $n!$ positive fixed points for all $\la\in P(n)$. 
The natural set of fixed points for $\mcI_{\la,\la}$ turns out to be the set 
of objects in $\mcP_{\la,\la}$ of the form $(S,S,\si)$; call this set
$\Surv_{\la}$. The next theorem proves there really are $n!$ such objects.
In fact, we require a bijective strengthening of this enumerative result
and a variation arising in connection with the local 
identity~\eqref{eq:local2} (rescaled by $n!$). For $n\geq 1$, define
\[ \CS_n=[n]\times [n-1]\times\cdots\times[3]\times[2]\times[1]
  =\{(c_n,\ldots,c_2,c_1): 1\leq c_k\leq k\mbox{ for }n\geq k\geq 1\}. \]
Clearly, $|\CS_n|=n!$. Elements of $\CS_n$ are called \emph{choice sequences}.

\begin{theorem}\label{thm:novel-n!}
(a) For each $\la\in P(n)$, there is a bijection 
$F_{\la}:\CS_n\rightarrow\Surv_{\la}$. 
\\ (b) Suppose we are given $\mu\in P(n)$ and a fixed (unlabeled)
rim-hook $\rho$ that may be removed from the border of $\dg(\mu)$ to leave 
a smaller partition shape.
Let $\Surv_{\mu,\rho}$ be the set of pairs $(T,\si)$ where $T$ is 
a RHT of shape $\mu$ whose last rim hook is $\rho$, 
and $\cycC(\si)$ equals the content of $T$.
There is a bijection $F_{\mu,\rho}:\CS_{n-1}\rightarrow \Surv_{\mu,\rho}$.
\end{theorem}
\begin{remark}\label{rem:novel-n!}
Part~(a) gives novel combinatorial interpretations of $n!$ parametrized
by the integer partitions of $n$. Since $\sgn(S)\sgn(S)=+1$, we can say that
there are $n!$ pairs $(S,\si)$ such that: $S$ is an \emph{unsigned} RHT
of shape $\la$ and some content $\be\in C(n)$; and $\cycC(\si)=\be$.
If we specify in advance the cells occupied by the last rim-hook in $S$,
then part~(b) says there are $(n-1)!$ such objects.
\end{remark}
\begin{proof}
As a preliminary step, we define a numbering of the rim-hooks that
can be removed from the border of the diagram of a partition $\nu\in P(k)$.
We number the cells in $\dg(\nu)$ with $1,2,\ldots,k$ in the reading
order, working from the top row down and reading left to right in each row.
In other words, the cell in row $i$, column $j$ of $\dg(\nu)$ receives
the number $\nu_1+\cdots+\nu_{i-1}+j$. Each cell in $\dg(\nu)$ corresponds
to a unique removable rim-hook on the border of $\dg(\nu)$, as explained at 
the start of the proof of Theorem~\ref{thm:local2}. 
For $1\leq c\leq k$, define the \emph{$c$th border rim-hook of $\nu$} 
to be that rim-hook corresponding to the $c$th cell of $\dg(\nu)$.
This numbering system is a bijection from $[k]$ to the set of
such border rim-hooks for $\nu$.

Next, we construct the bijection $F_{\la}$ by modifying
the counting argument in the proof of Lemma~\ref{lem:harmonic}(b).
Given a choice sequence $\cc=(c_n,\ldots,c_1)\in\CS_n$, we create
$F_{\la}(\cc)=(S,S,\si)\in\Surv_{\la}$ by choosing the rim-hooks
in $S$ working from the outside border inward and simultaneously
filling the cycles in the canonical cycle notation for $\si$
from right to left. To begin, remove the $c_n$th border rim-hook 
from $\dg(\la)$. This rim-hook will be the largest-labeled rim-hook in $S$.
Say $L$ cells were removed; then we are left
with the diagram of a smaller partition $\la^{(1)}\in P(n-L)$.
Build the rightmost cycle of $\si$ by starting with $1$,
then using entries $c_{n-1},c_{n-2},\ldots,c_{n-(L-1)}$ from $\cc$
to choose the remaining values in the cycle. Here and below,
an entry $c$ in the choice sequence means ``fill the next position
in the cycle with the $c$th smallest element not yet appearing in $\si$.''
To continue, remove the $c_{n-L}$th border rim-hook from
$\dg(\la^{(1)})$, say containing $L'$ cells, leaving the diagram
of some $\la^{(2)}$. Build the next-rightmost cycle of $\si$
by starting with the minimum element not used so far, then
using entries $c_{n-L-1},c_{n-L-2},\ldots,c_{n-L-(L'-1)}$ to choose
the remaining values in the cycle. Then remove the $c_{n-L-L'}$th
border rim-hook from $\dg(\la^{(2)})$, and continue similarly.
The passage from the choice sequence to the triple $(S,S,\si)$ is
reversible, as illustrated in the example below. 
So $F_{\la}$ is a bijection.

Part~(b) is proved in the same manner. Here we do not need $c_n\in [n]$
to choose the outermost rim-hook to remove from $\dg(\mu)$, since that
choice has been fixed in advance to be $\rho$. Thus, choice sequences
in $\CS_{n-1}$ suffice to construct uniquely each object in $\Surv_{\mu,\rho}$.
\end{proof}

\begin{example}\label{ex:get-choice}
In this example, we compute the choice sequence 
$F_{\la}^{-1}(S,S,\si)=(c_{18},c_{17},\ldots,c_1)$,
where $\la=(5,4,4,3,2)\in P(18)$, and $S$ and $\si$ are shown here:
\[ S = \yt{11344,1334,2344,255,25}, \qquad 
\si = (9,16,14)(7,13,15)(6,11,18,12)(2,17,3,10,8)(1,4,5). \]
To recover $c_{18}$, we use the numbering of the cells of $\dg(\la)$
shown here:
\[ \dg(\la)=\yt{12345,6789,{10}{11}{12}{13},{14}{*(lightgray)15}
   {*(lightgray)16},{17}{*(lightgray)18}} \]
We see that the highest-numbered rim-hook in $S$ (namely, the set
of cells containing $5$) is the $15$th border-rim hook of $\dg(\la)$, 
so $c_{18}=15$.  
To find $c_{17}$, we look at the second element in the rightmost
cycle of $\si$, namely $4$. The value $4$ is the third-smallest available
element (since $1$ has already been used), so $c_{17}=3$.
The next value in this cycle, namely $5$, is again the third-smallest
available element (since $1$ and $4$ have already been used), so $c_{16}=3$.
We abbreviate the reasoning in the last two sentences as follows:
\begin{align*}
c_{17}=3\mbox{ via } 
& (2,3,\underline{4},5,6,7,8,9,10,11,12,13,14,15,16,17,18);\\[0.1cm]
c_{16}=3\mbox{ via } 
& (2,3,\underline{5},6,7,8,9,10,11,12,13,14,15,16,17,18).
\end{align*}
To continue, we remove the rim-hook labeled $5$ to get a new
partition shape $\la^{(1)} = (5,4,4,1,1)$ with cells numbered
as follows:
\[  \dg(\la^{(1)})=\yt{123{*(lightgray)4}{*(lightgray)5},678
{*(lightgray)9},{10}{11}{*(lightgray)12}{*(lightgray)13},{14},{15}}.\]
The next highest rim-hook in $S$ (the set of cells containing $4$,
shaded gray in the preceding diagram) is the third
border rim-hook of $\la^{(1)}$, so $c_{15}=3$.
We proceed by examining the cycle $(2,17,3,10,8)$ in $\si$, which
necessarily begins with $2$. Reasoning as above, we deduce:
\begin{align*}
c_{14}=13\mbox{ via }
&(3,6,7,8,9,10,11,12,13,14,15,16,\underline{17},18);\\[0.1cm]
c_{13}=1 \mbox{ via } 
&(\underline{3},6,7,8,9,10,11,12,13,14,15,16,18);\\[0.1cm]
c_{12}=5 \mbox{ via }
&(6,7,8,9,\underline{10},11,12,13,14,15,16,18);\\[0.1cm]
c_{11}=3 \mbox{ via }
&(6,7,\underline{8},9,11,12,13,14,15,16,18).
\end{align*}
The computation continues similarly, as shown in the next figures.
\begin{center}
\begin{multicols}{2}
$\dg(\la^{(2)}) = \yt{12{*(lightgray)3},4{*(lightgray)5}{*(lightgray)6},7{*(lightgray)8},9,{10}}$ \columnbreak 
\[\begin{array}{ll}
c_{10}=2 & \text{and 6 begins the next cycle of $\si$;}\\[0.1cm]
c_9=3 & \text{ via } (7,9,\underline{11},12,13,14,15,16,18);\\[0.1cm]
c_8=8 & \text{ via } (7,9,12,13,14,15,16,\underline{18});\\[0.1cm]
c_7=3 & \text{ via } (7,9,\underline{12},13,14,15,16).
\end{array}\]
\end{multicols}
\end{center}
\begin{center}
\begin{multicols}{2}
$\dg(\la^{(3)}) = 
\yt{12,3,{*(lightgray)4},{*(lightgray)5},{*(lightgray)6}}$ \columnbreak 
\[\begin{array}{ll}
c_6=4 & \text{and $7$ begins the next cycle of $\si$;} \\[0.1cm]
c_5=2 & \text{ via } (9,\underline{13},14,15,16);\\[0.1cm]
c_4=3 & \text{ via } (9,14,\underline{15},16).
\end{array}\]
\end{multicols}
\end{center}
\begin{center}
\begin{multicols}{2}
$\dg(\la^{(4)}) = 
\yt{{*(lightgray)1}{*(lightgray)2},{*(lightgray)3}}$\columnbreak 
\[\begin{array}{ll}
c_3=1 & \text{and $9$ begins the next cycle of $\si$;} \\[0.1cm]
c_2=2 & \text{ via } (14,\underline{16});\\[0.1cm]
c_1=1 & \text{ via } (\underline{14}).
\end{array}\]
\end{multicols}
\end{center}
In summary, we have found
$F_{\la}^{-1}(S,S,\si) = (15,3,3,3,13,1,5,3,2,3,8,3,4,2,3,1,2,1)$.
Letting $\rho$ be the rim-hook of $S$ marked with $5$s, we also have
$F_{\la,\rho}^{-1}(S,\si)=(3,3,3,13,1,5,3,2,3,8,3,4,2,3,1,2,1)$.
\end{example}

\begin{remark}\label{rem:other-perm}
So far we have used permutations $\si\in S_n$ that permute the standard
set $\{1,2,\ldots,n\}$. The definition of $\cycC(\si)$ and all related
constructions still work if we replace this set by any given set of 
$n$ integers. This remark is needed below, when we operate on sub-permutations
of $\si$ where some of the cycles have been temporarily deleted.
\end{remark}

The involution $\mcI_{\la,\mu}$ acts on input $(S,T,\si)\in\mcP_{\la,\mu}$ 
as follows.  If $\la=\mu$ and the input is in $\Surv_{\la}$, then this
is a fixed point.  Otherwise let $\be=(\be_1,\ldots,\be_\ell)$ be the content 
composition for $S$ and $T$.  Remove the outermost rim-hook from $S$ 
(meaning the $\be_\ell$ cells in the filling of $\dg(\la)$ labeled $\ell$),
remove the outermost rim-hook from $T$ 
(meaning the $\be_\ell$ cells in the filling of $\dg(\mu)$ labeled $\ell$),
and remove the rightmost cycle in the canonical cycle notation of $\si$.
Do this repeatedly until reaching a survivor object $(S^0,S^0,\si^0)$
for some shape $\ga$, where $\ga$ might be empty. 
Let $(S',T',\si')$ be the immediately preceding object.
We go from $S^0$ to $S'$ by adding a rim-hook $\eta$ of size $\be_k$
to get an object of some shape $\bla$.
We go from $S^0$ to $T'$ by adding a rim-hook $\rho$ of size $\be_k$
to get an object of some shape $\bmu$.
Here, $\si^0$ consists of the leftmost $k-1$ cycles in $\si$,
while $\si'$ consists of the leftmost $k$ cycles in $\si$.
By construction, $(T',\si')\in\Surv_{\bmu,\rho}$, 
and the set $G(\bla,\bmu)$ (defined in Theorem~\ref{thm:local2})
is nonempty. The proof of that theorem shows that $G(\bla,\bmu)$ consists
of exactly two oppositely-signed partitions and describes how to go from
one of these objects to the other. Let $\tilde{\ga}$ be the other object
appearing with $\ga$ in $G(\bla,\bmu)$, so that 
$\tilde{\eta}=\bla\setminus\tilde{\ga}$ and
$\tilde{\rho}=\bmu\setminus\tilde{\ga}$ are rim-hooks of the same size.

To build $\mcI_{\la,\mu}(S,T,\si)$, first compute 
$(T'',\si'')=F_{\bmu,\tilde{\rho}}\circ F^{-1}_{\bmu,\rho}(T',\si')
\in\Surv_{\bmu,\tilde{\rho}}$ using the bijections
from Theorem~\ref{thm:novel-n!}(b), as modified in Remark~\ref{rem:other-perm}.
Get $S''$ by copying all values in $T''$ in the sub-shape $\tilde{\ga}$,
then filling the rim-hook $\tilde{\eta}$ with the next unused value $\tilde{k}$.
To finish, restore the previously removed rim-hooks
of lengths $\be_{k+1},\ldots,\be_{\ell}$
to the fillings $S''$ and $T''$ (respectively)
in the same positions they occupied in $S$ and $T$.
Here the labels (originally $k+1,k+2,\ldots$) of the restored
rim-hooks get renumbered to be $\tilde{k}+1,\tilde{k}+2,\ldots$.
Append to the right end of $\si''$ the
cycles previously removed from the right end of $\si$.
The resulting triple is $\mcI_{\la,\mu}(S,T,\si)$.

Acting by $\mcI_{\la,\mu}$ again reverses all the steps and returns us to 
$(S,T,\si)$, so $\mcI_{\la,\mu}$ is an involution on $\mcP_{\la,\mu}$. 
The fixed point set of $\mcI_{\la,\la}$ in $\mcP_{\la,\la}$ has size $n!$, 
by Theorem~\ref{thm:novel-n!}(a).
This completes our bijective proof of $A_n\bB_n=n!I_{P(n)}$.

\begin{example}
Let $\la=(4,3,2,1)$, $\mu=(4,3,3)$.  
We compute $\mcI_{\la,\mu}(S,T,\si)$ for the following objects:
\[ S = \yt{1133,123,22,2}, \quad T = \yt{1122,122,333},
\quad \si = (5,8,6)(2,10,9,4)(1,3,7).  \]
Here, $\be=(3,4,3)$,
$\sgn(S) = -1\cdot 1\cdot -1 = 1$, and $\sgn(T) = -1\cdot -1\cdot 1 = 1$. 
To reach a survivor $(S^0,S^0,\si^0)$,
we remove the rim-hooks labeled 3 and 2 from $S$ and $T$,
and we remove the two rightmost cycles of $\si$.
This produces $\si^0=(5,8,6)$ and $S^0 = \yt{11,1}$ of shape $\ga=(2,1)$. 
We construct $S'$ and $T'$ by restoring 
the rim-hooks labeled 2 as shown here:
\[
S^0 = \yt{11,1} \xrightarrow{\text{add }\eta} S' = \yt{11,1{*(lightgray) 2},{*(lightgray) 2}{*(lightgray) 2},{*(lightgray) 2}}; \qquad 
S^0 = \yt{11,1} \xrightarrow{\text{add }\rho} 
T' = \yt{11{*(lightgray) 2}{*(lightgray) 2},1{*(lightgray) 2}{*(lightgray) 2}}.
 \]
We find that $\sigma' = (5,8,6)(2,10,9,4)$, 
$S'$ has shape $\bla=(2,2,2,1)$ and $T'$ has shape $\bmu = (4,3)$. 
By using the bijection in the proof of Theorem \ref{thm:local2}
(cf.~Example \ref{ex:local2}), we find $\tga = (2,2)$. 
This defines $\tilde{\eta}$ and $\tilde{\rho}$, which are shaded in gray here:
\boks{0.3}
\[
\tilde{\eta} = \bla\smin \tga = 
\y{2,2}*[*(lightgray)]{0,0,2,1},\qquad
\tilde{\rho} = \bmu\smin \tga = 
\y{2,2}*[*(lightgray)]{2+2,2+1}.
\]
\boks{0.4}

Next we compute the choice sequence $F^{-1}_{\bmu,\rho}(T',\sigma')$,
which is the sequence $F^{-1}_{\bmu}(T',\sigma')$ with its first entry $c_7$
deleted. One can check that $c_7=2$ since $\rho$ is the second
border rim-hook of $\bmu$, but we do not need this value.
Note that $\si'$ permutes the set $\{2,4,5,6,8,9,10\}$,
and the rightmost cycle of $\si'$ begins with the smallest value in 
this set, namely $2$. Looking at the remaining entries in this cycle,
we recover $c_6$, $c_5$, $c_4$ as follows (using the same notation
from Example~\ref{ex:get-choice}):
\[\begin{array}{ll}
c_6=6 & \text{ via } (4,5,6,8,9,\underline{10});\\[0.1cm]
c_5=5 & \text{ via } (4,5,6,8,\underline{9});\\[0.1cm]
c_4=1 & \text{ via } (\underline{4},5,6,8).
\end{array}\]
Removing the cells containing $2$ from $T'$ leaves the shape $(2,1)$.
The rim-hook labeled $1$ in $T'$ is the first border rim-hook of this shape,
so $c_3=1$. We deduce the rest of the choice sequence as follows:
\[\begin{array}{ll}
c_{3}=1 & \text{and 5 begins the next cycle of $\si'$;}\\[0.1cm]
c_2=2& \text{ via } (6,\underline{8});\\[0.1cm]
c_1=1 & \text{ via } (\underline{6}).
\end{array}\]

So far, we have $F_{\bmu}^{-1}(T',\sigma') = (2,6,5,1,1,2,1)$ 
and $\mathbf{c}=F_{\bmu,\rho}^{-1}(T',\sigma')= (6,5,1,1,2,1)$.  
The numbered diagram $\yt{12{*(lightgray) 3}{*(lightgray) 4},
56{*(lightgray) 7}}$ shows that $\tilde{\rho}$ is the third border
rim-hook of $\bmu$. Therefore, setting $\mathbf{c'}=(3,6,5,1,1,2,1)$,
we have $(T'',\si'')=F_{\bmu, \tilde{\rho}}(\textbf{c})=F_{\bmu}(\textbf{c}')$.
We compute $T''$ and $\si''$ as follows.
Like $\si'$, $\si''$ permutes the set $\{2,4,5,6,8,9,10\}$ 
and the rightmost cycle of $\si''$ begins with $2$.
This cycle contains $|\tilde{\rho}|=3$ values. We use $c'_6=6$
and $c'_5=5$ to determine these values, as follows:
\[\begin{array}{ll}
(2,10,\blk) &\text{ via } (4,5,6,8,9,\underline{10})\text{ and } c_6' = 6;\\[0.1cm]
(2,10,9) &\text{ via } (4,5,6,8,\underline{9})\text{ and } c_5' = 5.
\end{array}\]
The next entry $c_4'=1$ tells us which rim-hook to remove from the
remaining shape $(2,2)$, as shown by the gray cells in the numbered
diagram $\yt{1{*(lightgray) 2},{*(lightgray) 3}{*(lightgray) 4}}$.
We begin the next cycle of $\si''$ with the least available value $4$.
The remaining two values in this cycle are found as follows:
\[\begin{array}{ll}
(4,5,\blk) &\text{ via } (\underline{5},6,8) \text{ and } c_3' = 1;\\[0.1cm]
(4,5,8) &\text{ via } (6,\underline{8}) \text{ and } c_2' = 2.
\end{array}\]
We conclude by using $c_1'=1$ to remove the first border rim-hook
of the remaining shape $(1)$, as shown here: $\yt{{*(lightgray) 1}}$.
We use the least remaining value $6$ to make the leftmost cycle $(6)$ in 
$\si''$.  We find $T''$ by labeling the rim-hooks removed from $\bmu$ 
using labels $3$, $2$, $1$.  We find $S''$ by copying the rim-hooks of $T''$
in the shape $\tga$ and labeling the cells of $\tilde{\eta}$ with 3.
In summary, we have found
\[ S'' = \yt{12,22,33,3},\quad T'' = \yt{1233,223},\quad
\si'' = (6)(4,5,8)(2,10,9).  \]
Finally, we compute $\mcI_{\lambda,\mu}(S,T,\sigma)$ by restoring 
to $S''$ and $T''$ the rim-hooks originally labeled $3$ in $S$ and $T$ 
(using $4$ as their new label) and appending the removed 
cycle $(1,3,7)$ to $\si''$. This produces
\[ (S^*, T^*, \sigma^*) = \mcI_{\lambda,\mu}(S,T,\sigma) = 
\left(\yt{1244,224,33,3},\ \yt{1233,223,444},\ (6)(4,5,8)(2,10,9)(1,3,7)\right).
\]
Note that $\sgn(S^*) = 1\cdot -1\cdot -1 \cdot -1 = -1$ 
and $\sgn(T^*) = 1\cdot -1 \cdot -1 \cdot 1 = 1$, so 
$\sgn(S^*)\sgn(T^*) = - \sgn(S)\sgn(T)$ as needed. 
Acting on $(S^*,T^*,\si^*)$ by $\mcI_{\la,\mu}$ reverses all 
the preceding steps and recreates the original triple $(S,T,\si)$.
\end{example}

 
\end{document}